\documentclass[a4paper,fleqn]{cas-sc}

\usepackage[authoryear,longnamesfirst]{natbib}
\usepackage{algorithm}  
\usepackage{algorithmicx}  
\usepackage{algpseudocode} 
\usepackage{undertilde}
\usepackage{tikz}
\usepackage{bm}
\usepackage{float}
\usepackage{colortbl}
\def\tsc#1{\csdef{#1}{\textsc{\lowercase{#1}}\xspace}}
\tsc{WGM}
\tsc{QE}
\tsc{EP}
\tsc{PMS}
\tsc{BEC}
\tsc{DE}
\newcommand{\udots}{\mathinner{\mskip1mu\raise1pt\vbox{\kern7pt\hbox{.}}  
		\mskip2mu\raise4pt\hbox{.}\mskip2mu\raise7pt\hbox{.}\mskip1mu}} 
\newtheorem{assumption}{Assumption}[section]
\newtheorem{theorem}{Theorem}
\newproof{proof}{Proof}
\newtheorem{remark}{Remark}

\begin{document}
\let\WriteBookmarks\relax
\def\floatpagepagefraction{1}
\def\textpagefraction{.001}

\title [mode = title]{Meta-MgNet: Meta Multigrid Networks for Solving Parameterized Partial Differential Equations}               \shorttitle{Meta Multigrid Networks for Solving Parameterized PDEs}

\author[1]{Yuyan Chen}
\address[1]{School of Mathematical Science, Peking University}
\ead{chenyuyan@pku.edu.cn}

\author[2]{Bin Dong}
\address[2]{Beijing International Center for Mathematical Research \&\\Institute for Artificial Intelligence, Peking University}
\ead{dongbin@math.pku.edu.cn}

\author[3]{Jinchao Xu}
\address[3]{Department of Mathematics, Pennsylvania State University}
\ead{xu@math.psu.edu}

\begin{abstract}
This paper studies numerical solutions for parameterized partial differential equations (PDEs) with deep learning. Parametrized PDEs arise in many important application areas, including design optimization, uncertainty analysis, optimal control, and inverse problems. The computational cost associated with these applications using traditional numerical schemes can be exorbitant, especially when the parameters fall into a particular range, and the underlying PDE model is required to be solved with high accuracy using a fine spatial-temporal mesh. Recently, solving PDEs with deep learning has become an emerging field in scientific computing. Existing works demonstrate great potentials of the deep learning-based approach in speeding up numerical solutions of various types of PDEs. However, there is still limited research on the deep learning approach for solving parameterized PDEs. If we directly apply existing deep supervised learning models to solving parameterized PDEs, the models need to be constantly fine-tuned or retrained when the parameters of the PDE change. This limits the applicability and utility of these models in practice. To resolve this issue, we propose a meta-learning based method that can efficiently solve parameterized PDEs with a wide range of parameters without retraining. Our key observation is to regard training a solver for the parameterized PDE with a given set of parameters as a learning task. Then, training a solver for the parameterized PDEs with varied parameters can be viewed as a multi-task learning problem, to which meta-learning is one of the most effective approaches. This new perspective can be applied to many existing PDE solvers to make them suitable for solving parameterized PDEs. As an example, we adopt the Multigrid Network (MgNet) \citep{MgNet} as the base solver. To achieve multi-task learning, we introduce a new hypernetwork, called Meta-NN, in MgNet and refer to the entire network as the Meta-MgNet. Meta-NN takes the differential operators and the right-hand-side of the underlying parameterized PDEs as inputs and generates appropriate smoothers for MgNet, which are essential ingredients for multigrid methods and can significantly affect the convergent speed. The proposed Meta-NN is carefully designed so that Meta-MgNet has guaranteed convergence for Poisson's equation. Finally, extensive numerical experiments demonstrate that Meta-MgNet is more efficient in solving parameterized PDEs than the MG methods and MgNet trained by supervised learning.
\end{abstract}

\begin{graphicalabstract}
\includegraphics{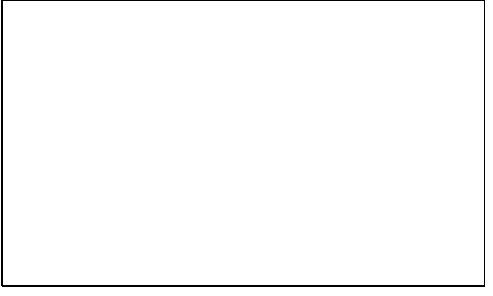}
\end{graphicalabstract}

\begin{highlights}
\item Viewing solving parameterized PDEs as multi-task learning.
\item Introducing meta-learning to solve parameterized PDEs and designing a new data-driven solver called Meta-MgNet.
\item Guaranteed convergence of the proposed Meta-MgNet for Poisson's equation.
\end{highlights}

\maketitle
\section{Introduction}
Partial differential equations are essential tools in many areas, such as physics, chemistry, biology, and economics. Most PDEs we encounter in practice contain parameters representing the system's physical or geometric properties, e.g., kinetic characteristics, material properties, the shape of the domain, etc. In practice, we often found ourselves in multi-query scenarios where the PDEs need to be solved for numerous different parameters with high accuracy and efficiency. Such scenarios include design optimization, optimal control, uncertainty quantification, inverse problems, etc. Therefore, a uniformly efficient solver for all parameters is urgently needed. 

In this paper, we consider the following parameterized steady-state PDEs
\begin{equation}\label{pde}
\left\{
\begin{split}
\utilde {\mathcal A}_{\eta} \  \utilde u &= \utilde f,&\text{ in } \Omega,\\
\utilde u &= \utilde u_b,&\text{ on } \partial \Omega,
\end{split}
\right.
\end{equation}
where $\Omega \subset \mathbb R^d, d, n\in \mathbb N_+$, $\mathfrak U$, $\mathfrak F$ are two function space on $\Omega$, and $\mathfrak U_b$ is a function space on $\partial \Omega$, $\utilde u = (u^1, u^2,..., u^n)\in \mathfrak U^n$, $\utilde f = (f^1, f^2,..., f^n)\in\mathfrak F^n$, $\utilde u_b = (u_b^1, u_b^2,..., u_b^n)\in \mathfrak U_b^n$. And $\utilde {\mathcal A}_{\eta}:\mathfrak U^n\rightarrow\mathfrak F^n$ is a linear differential operator with parameter $\eta=(\eta_1,\ldots,\eta_m)$. For convenience, we will omit $\eta$ when there is not any confusion. In this paper, the specific PDEs we choose for our numerical experiments are 2D/3D steady-state anisotropic diffusion equations and Ossen equations. The steady-state diffusion equations are widely used in fluid mechanics, electronic science, image processing, etc. The Ossen equations play an important role in fluid mechanics. We recall these PDEs as follows:
\begin{enumerate}[(1)]
	\item 2D anisotropic diffusion equations: \begin{equation*}
	\left\{
	\begin{split}
	-\nabla\cdot(C \nabla u) &= f, &\text{ in } \Omega,\\
	u &= 0,  &\text{ on } \partial \Omega,
	\end{split}
	\right.
	\end{equation*}
	where $C = C(\epsilon, \theta)= \begin{pmatrix}
	\cos \theta& -\sin\theta\\\sin\theta&\cos \theta
	\end{pmatrix}
	\begin{pmatrix}
	1& 0\\0&\epsilon
	\end{pmatrix}
	\begin{pmatrix}
	\cos \theta& \sin\theta\\-\sin\theta&\cos \theta
	\end{pmatrix}$ is a $2\times 2$ matrix, $\epsilon<1, \theta\in[0,\pi].$
	\item 3D anisotropic diffusion equations:\begin{equation*}
	\left\{
	\begin{split}
	-\nabla\cdot(C \nabla u) &= f, &\text{ in } \Omega,\\
	u &= 0,  &\text{ on } \partial \Omega,
	\end{split}
	\right.
	\end{equation*}
	on domain $\Omega = [0,1]^3$. In this paper, we only concern the case with $$C = 
	\begin{pmatrix}
	\epsilon_0& &\\&\epsilon_1&\\&&\epsilon_2
	\end{pmatrix},\quad \epsilon_0, \epsilon_1, \epsilon_2>0.
	$$
	\item Ossen equations:
	\begin{equation*}
		\left\{\begin{aligned}
			-\mu \Delta \utilde u+(\utilde a \cdot \nabla)\utilde u + \nabla p &=  \utilde f,  &\text{in } \Omega,\\
			-\text{div} \utilde u &= \utilde 0,  &\text{in } \Omega,\\
			\utilde u &= \utilde 0,    &\text{on }  \partial \Omega.
		\end{aligned}\right.
	\end{equation*}
	where $\mu = \dfrac 1{Re}$ and $Re$ is Reynold number, and $\utilde a = (a_x, a_y)^\top$.
\end{enumerate}

When discretized, equation \eqref{pde} becomes a linear system of equations
\begin{equation}\label{pde:discrete}
\mathbf A_\eta \mathbf u = \mathbf f.
\end{equation}
Linear system \eqref{pde:discrete} is usually of a very large scale in practice, and iterative methods are often used to solve it. The multigrid (MG) method \citep{xu1992iterative,xu2002method, hackbusch2013multi} is one of the most compelling classical methods. The computational complexity of the MG method for the Poisson's equation is only $O(n)$, where $n$ is the size of the matrix; as compared to another popular high-performance numerical method, the spectral method, whose complexity is $O(n\text{log}n)$. However, the MG method still has trouble solving PDEs \eqref{pde} with $\eta$ falling into a specified range. Taking the 2D anisotropic diffusion equation as an example \citep{yu2013analysis}, the computational cost of the MG method grows rapidly with $\epsilon\rightarrow0$ (see \textbf{Figure \ref{fig:1}}).
\begin{figure}
    \centering
    \includegraphics[width = 0.4\textwidth]{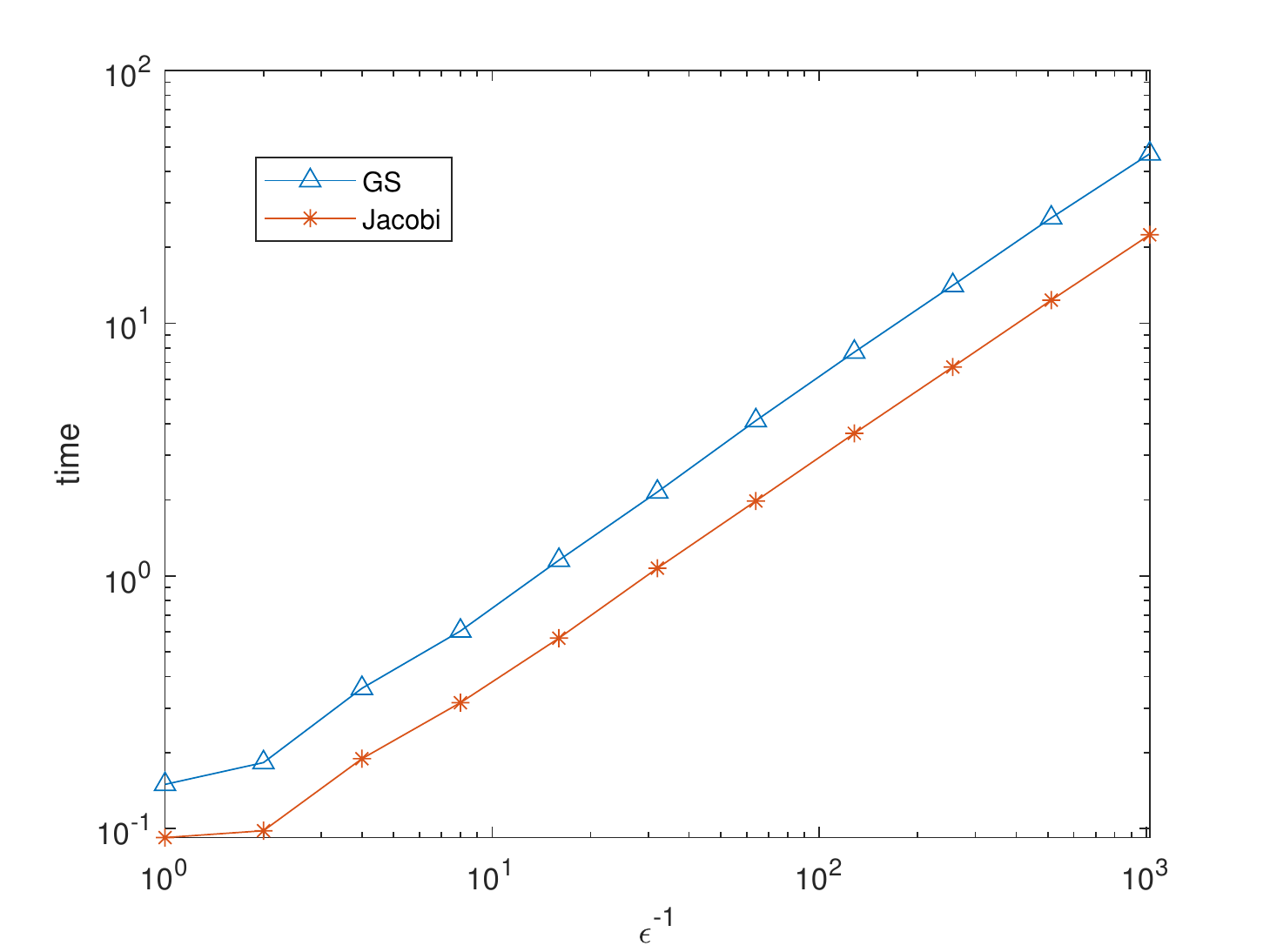}
    \caption{The wall times of convergence grows rapidly with $\epsilon\rightarrow0$.}
    \label{fig:1}
\end{figure}

Therefore, people try to adjust the parameters and components in the MG method (such as smoothers, prolongations, and restrictions) to improve its performance according to $\eta$. Although decades of continuing researches are devoted to the convergence theory of the MG method to find the theoretically best parameters, the computational cost of finding such optimal parameters for a given $\eta$ can be much higher than solving the linear system \eqref{pde:discrete} itself. For example, \citep{xu2017algebraic} derives the theoretically optimal prolongation for a given smoother. However, it requires solving an eigenvalue problem. Crucial parameters such as the damped coefficient of damped Jacobi smoother and the prolongations and restrictions of the algebraic MG method are mostly designed by human experts. However, these manually designed parameters can be rather complicated and have limited improvements for MG methods in practice. 

Before the rise of deep learning, machine learning had a similar challenge as MG methods. Classical machine learning heavily relies on feature engineering, where people tried to manually design various types of features that are later fed into a classification or regression model. However, the quality of the features depends on the data set, the downstream task, and also the choice of the classification and regression model. It is extremely difficult to design fully adaptive and good feature extractors purely based on the human experience. This is where deep learning has been proven tremendously effective. After the advent of deep learning, feature extractors can be learned directly from data \citep{lecun2015deep} in an end-to-end fashion. This often leads to feature extractors significantly surpass the previous ones designed by human experts in performance. This enables deep learning to achieve enormous success in many areas of artificial intelligence, such as natural language processing \citep{hinton2012deep} and games\citep{silver2016mastering}. 

The success of deep learning motivated us to resort to a data-driven approach to determine MG methods' parameters. Among all deep learning approaches, deep supervised learning has become one of the most popular data-driven approaches in scientific computing. Therefore, using deep supervised learning to determine the parameters of the MG method is a natural choice. Since classical MG methods do not work well for certain $\eta$, we can learn the parameters in the MG method to make it uniformly efficient for all $\eta$. 

Deep supervised learning can be easily applied to the MG method if we cast the MG method into a similar formulation as deep neural networks. As first observed in \citep{MgNet}, the MG method is an iterative method which can be viewed as a feedforward network. Furthermore, the prolongations, restrictions, and some special smoothers can all be written as convolutions. Therefore, the MG method has a natural connection with the convolutional neural network (CNN). With that, \citep{MgNet} introduced the Multigrid Network (MgNet). 

The original MgNet in \citep{MgNet} was proposed for image classification. In this paper, we convert it into a form suitable to solve PDEs and refer to it as PDE-MgNet. PDE-MgNet takes the right-hand-side function $\utilde f$ as input and the approximated solution $\utilde u$ as output. PDE-MgNet performs very well when it is trained on a data set generated by $\eta\sim \mathcal{P}$ in PDEs \eqref{pde} and tested on some other $\eta$ generated from the same distribution $\mathcal{P}$. However, it may generalize poorly (i.e., convergence slows down significantly) beyond the training setting, i.e., for $\eta\sim\mathcal{P}'$ with $\mathcal{P}'$ different from $\mathcal{P}$.

This problem is common for deep supervised learning. To resolve it, we need to adopt a more robust learning mechanism. In this paper, we regard learning a solver for PDEs \eqref{pde} with a given $\eta$ as one learning task. Then learning solvers for PDEs \eqref{pde} for all $\eta$ can be naturally viewed as a multi-task learning problem. In the area of artificial intelligence, meta-learning is an effective approach to solve multi-task learning problems. Thus we propose to use the meta-learning approach to improve PDE-MgNet. Note that PDE-MgNet is just an example we choose in this paper. A similar methodology can be applied to enhance other numerical solvers for parameterized PDEs.

First, let us briefly review meta-learning. Meta-learning \citep{Thrun1998LearningTL, hospedales2020metalearning}, also known as learning to learn, is a science of studying how different machine learning approaches perform on a wide range of tasks and use these experiences to speed up the learning of new tasks. In contrast, supervised learning trains a model $\mathrm M_\eta$ for each task $\mathrm T_\eta$ separately. Even though all $\mathrm M_\eta$ have the same network structures, we need to retrain the model each time $\eta$ changes. If there are numerous tasks to learn, the computation cost can be unbearable for supervised learning. The meta-learning approach resolves this problem by gaining experience over multiple learning episodes - often covering the distribution of related tasks - and uses this experience to improve its future learning performance.

There are several strategies in meta-learning. We briefly introduce two strategies that are suitable for our task of interest. Suppose the only difference between $\mathrm M_\eta$ is the weights of neural networks, namely $$\text{M}_\eta = \text{M}(\mathbf x; \mathbf w_\eta),$$
 where $\mathbf x$ is the input and $\mathbf w_\eta$ is the weights of neural networks. The two strategies to quickly find suitable $\mathbf{w}_\eta$ for each task $\eta$ are given as follows. 
\begin{enumerate}[(1)]
	\item Finding a good initialization \citep{Model-Agnostic_Meta-Learning_for_Fast_Adaptation_of_Deep_Networks, On_First-Order_Meta-Learning_Algorithms}:
	This strategy makes use of a series of tasks $\text{T}_{\eta_1}, \text{T}_{\eta_2},..., \text{T}_{\eta_n} $ to obtain a good initial weights $\mathbf w_0$ of the deep neural network $\text{M}$. Then, we can easily get  $\mathbf w_\eta$ for each task $\eta$ by fine-tuning from $\mathbf w_0$. For example, suppose the task $\mathrm T_\eta$ is sampled from distribution $\mathcal{P}(\mathrm T)$, the loss of model $\mathrm M$ on the task $\mathrm T$ is $L_{\mathrm T}(M)$, and we use gradient descent with learning rate $\alpha$ to train the model. After one step of gradient descent starting from $\mathbf w_0$, we obtain the updated weights as $\mathbf w' = \mathbf w_0 - \alpha\nabla_{\mathbf w}L_{\mathrm T_\eta}(\mathrm M_{\eta}(\cdot;\mathbf w_0 ))$.
	Thus, the initial weights $\mathbf w_0$ should minimize the following expectation
	\begin{equation}\label{MAML}
	\mathop{\mathbb	E}_{\mathrm T_\eta \sim \mathcal{P}(\mathrm T)}L_{\mathrm T_\eta}(\mathrm M_{\eta}(\cdot; \mathbf w' )).
	\end{equation}
	This is the main idea of Model-Agnostic Meta-Learning (MAML) proposed by \citep{MetaPruning_Meta_Learning_for_Automatic_Neural_Network_Channel_Pruning}, where an algorithm solving \eqref{MAML} is also proposed to estimate $\mathbf w_0$.

	\item Hypernetwork \citep{HyperNetworks, Stochastic_Hyperparameter_Optimization_through_Hypernetworks, One-Shot_Model_Architecture_Search_through_HyperNetworks, zhang2020metainvnet}: This strategy relies on the designs of a network called hypernetwork to infer $\mathbf w_\eta$ instead of direct learning of $\mathbf w_\eta$ by training. The hypernetwork takes the information on the task $\eta$ as input and $\mathbf w_\eta$ as output. The hypernetwork are trained to make an accurate prediction on $\mathbf w_\eta$ for each task $\eta$. When the hypernetworks are not powerful enough to make accurate predictions on $\mathbf w_\eta$, we treat the approximation of $\mathbf w_\eta$ as a good initial value and resort to fine-tuning on each task $\eta$ to improve the prediction. In particular, if the output of a hypernetwork is the same for all $\eta$, we can think that the hypernetwork gives a good initialization for all tasks $T_\eta$. In this regard, the previous strategy can be viewed as a special case of the hypernetwork approach.  
\end{enumerate}

By viewing solving the parameterized PDE \eqref{pde} for a given $\eta$ as a task $\mathrm T_\eta$, we adopt the hypernetwork based meta-learning approach to improve upon the PDE-MgNet. With this, we introduce a new model called Meta-MgNet. Compared to PDE-MgNet, the Meta-MgNet uses a carefully designed hypernetwork to infer the model parameters of the MgNet, instead of learning it directly from data. The hypernetwork grants the Meta-MgNet great ability of in-distribution generalization and out-of-distribution transfer. We shall call this hypernetwork Meta-NN. The Meta-NN is used to infer parameters of specific components in the MgNet. In this paper, we select two types of smoother as an example and design the corresponding meta-smoother (i.e., using Meta-NN to infer parameters of the smoother) for the Meta-MgNet. The two types of smoother are the convolution smoother, which is exactly what MgNet uses, and the smoother based on subspace correction. For the convolution smoother, the Meta-NN infers its convolution kernels. For the smoother based on subspace correction, the Meta-NN infers the spanning vectors of the subspace. The parameters of the Meta-NN is first trained on a data set with mixed data from different $\eta$. Then, we can fine-tune the Meta-NN while solving a specific $\mathrm T_\eta$. However, our numerical experiments show that the weights given by Meta-NN without retraining are good enough. Therefore, we shall skip the retraining step in the experiments.

\subsection{Related Work}
In the area of solving PDEs by machine learning, especially deep learning, the existing algorithms can be divide into the following two categories. 
\begin{enumerate}[(1)]
	\item Using neural networks to approximate the function $\utilde u$: These algorithms are suitable for a PDE with a fixed differential operator $\utilde{A}$ and the right-hand side $\utilde f$. The most notable advantage of this approach is that they can: 1) overcome the curse of dimensionality and solve high-dimension PDEs; 2) resolve complex geometries in the solution due to the meshless representation of neural networks. Successful examples include the nonlinear convection-diffusion PDEs \citep{2017arXiv171110561R, 2017arXiv171110566R}, Riemann Problem \citep{2019arXiv190412794M}, high-dimension PDEs\citep{Beck2019, E2018, Han8505, SIRIGNANO20181339}, and others \citep{xu2020dlga, zhang2020PINN, WANG2020108963, liu2020multi, sheng2020pfnn}. Nevertheless, if the parameter $\eta$ or the right-hand-side function $\utilde f$ is changed, the neural network often needs to be retrained. 
	\item  Using neural networks to approximate the solution operator $\utilde{\mathcal A}^{-1}$: The general modeling strategy of methods in this category is to replace a portion of the classical numerical method with neural networks to improve its performance. For example, \citep{RAY2019108845} uses NN to estimate locations of discontinuous, \citep{Discacciati:263616} uses NN to introduce an appropriate amount of artificial dissipation in the numerical solver. There are also works using NN to approximate the entire operator $\utilde{\mathcal A}^{-1}$. For example, \citep{LearnSolver} trains a U-Net as a solver for Poisson's equations. The most related work to the current one is \citep{MetaLearningPDE}, where the authors use a meta-learning approach for parameterized pseudo-differential operators with deep neural networks. However, there are two main differences between their work and ours. Firstly, their idea is to find the map $\eta \mapsto \utilde{\mathcal A}^{-1}_\eta$ directly. Ours is based on the observation that learning solvers for PDE \eqref{pde} with different $\eta$ is a multi-task learning problem and then introduce meta-learning to solve it. Secondly, their approach is based on the wavelet method, while ours is based on the MgNet.
	
	There are also several studies on improving the MG method by deep-learning. For instance, \citep{katrutsa2017deep, greenfeld2019learning, luz2020learning} proposed a series of data-driven approaches to design prolongations and restrictions in MG. However, these methods are based on supervised learning while we focus on learning smoothers based on meta-learning.  
\end{enumerate} 

In addition to the approaches above, the reduced-order modeling (ROM) \citep{ReducedOrder2009, ModelReductionSurvey} is also a widely used method for solving parameterized PDEs. The objective of ROM is to generate reduced models that are cheaper to solve while still well approximate the original PDEs. For example, \citep{ModelReduction2020} proposes a novel framework for projecting dynamical systems onto nonlinear manifolds using minimum-residual formulations at the time-continuous and time-discrete levels; \citep{fresca2020comprehensive} proposes new deep learning-based nonlinear ROM.

The remaining sections are organized as follows. In section 2, we will introduce necessary notations for the rest of the paper and discuss how to represent discrete PDEs by convolutions. In section 3, we review the multigrid method and the PDE-MgNet. In section 4, we introduce the proposed Meta-MgNet and provide a preliminary convergence analysis of the algorithm. In section 5, we present the numerical experiments and comparisons among the classical MG methods, PDE-MgNet, and Meta-MgNet. Concluding remarks are given in section 6.

\section{Convolutions and Differential Operators}
The key to solving PDEs is to discretize the differential operators properly. The main goal of this section is to present a theorem that convolutions can express the most meaningful discretizations of differential operators. This theorem plays an essential role for us to rewrite traditional numerical solvers as CNNs. The MgNet introduced by \citep{MgNet} is an example, which is a reformulation from the MG method. Furthermore, we will introduce the definition of convolution operators and then show how to use convolutions to represent discrete forms of differential operators.

\subsection{Convolution Operators}
In this paper, we only consider the convolution of two second-order tensors and the convolution of a fourth-order tensor and a third-order tensor. Consider two second order tensors $\mathsf K = (\mathsf K_{\jmath,\imath})$, with $\jmath,\imath \in \mathbb Z$ and $\mathsf v = (\mathsf v_{j, i})$, with $ i \in \{0,1,...,I\}, j\in\{0,1,...,J\}$, i.e.
	\begin{equation*}
	\mathsf K = \begin{pmatrix}
	\ddots&&\vdots&&\udots\\
	&\mathsf K_{-1,-1} &\mathsf K_{-1,0}& \mathsf K_{-1,1}&\\
	\cdots&\mathsf K_{0,-1} &\mathsf K_{0,0}& \mathsf K_{0,1}&\cdots\\
	&\mathsf K_{1,-1} &\mathsf K_{1,0}& \mathsf K_{1,1}&\\
	\udots&&\vdots&&\ddots\\
	\end{pmatrix}\quad\mbox{and}\quad
	\mathsf v =
	\begin{pmatrix}
	\mathsf v_{0,0} &\mathsf v_{0,1}&\cdots&\mathsf v_{0,I}\\
	\mathsf v_{1,0} &\mathsf v_{1,1}&\cdots&\mathsf v_{1,I}\\
	\vdots&\vdots&\ddots&\vdots\\
	\mathsf v_{J,0} &\mathsf v_{J,1}&\cdots&\mathsf v_{J,I}
	\end{pmatrix}.
	\end{equation*}

	The convolution $\mathsf K\star \mathsf v$ is also a second-order tensor, and the values of its components are
	\begin{equation*}
	(\mathsf K \star \mathsf v)_{j,i} = \sum _{\jmath,\imath\in \mathbb Z} \mathsf K_{\jmath,\imath} \mathsf v_{j+\jmath,i+\imath},
	\end{equation*}
	where $i \in \{0,1,...,I\}, j\in\{0,1,...,J\}$. If $i+\imath\notin \{0,1,...,I\}$ or $j+\jmath\notin \{0,1,...,J\}$, we set $\mathsf v_{j+\jmath,i+\imath}=0$.
	
	Consider a forth-order tensor $\mathsf K$ and a third-order tensor $\mathsf v$, where $\mathsf K = (\mathsf K_{l,k,\jmath,\imath})$, with $l,k \in \{1,2,...,S\}, \jmath,\imath \in \mathbb Z$ and $\mathsf v = (\mathsf v_{k, j, i})$, with $k \in \{1,2,...,S\}, i \in \{0,1,...,I\}, j\in\{0,1,...,J\}$. The definitions of convolution, convolution with stride and deconvolution (the transpose of convolution) are listed below
	\begin{enumerate}[(1)]
	\item \textbf{Convolution}: The convolution $\mathsf K\star \mathsf v$ is a third order tensor and the value of its components are
	\begin{equation*}
	(\mathsf K \star \mathsf v)_{l,j,i} =\sum_{k=1}^S \sum _{\jmath,\imath\in \mathbb Z} \mathsf K_{l,k ,\jmath,\imath} \mathsf v_{k , j+\jmath,i+\imath}.
	\end{equation*}
	\item \textbf{Convolution with strides}: Suppose $I_s, J_s \in \mathbb N^+$, and we use $\star_{J_s,I_s}$ to express the convolution with stride$=(J_s, I_s)$ and the components of $\mathsf K \star_{J_s,I_s} \mathsf v$ are
	\begin{equation*}
	(\mathsf K \star_{J_s,I_s} \mathsf v)_{l,j,i} =\sum_{k=1}^S \sum _{\jmath,\imath\in \mathbb Z} \mathsf K_{l,k ,\jmath,\imath} \mathsf v_{k , jJ_s+\jmath,iI_s+\imath}.
	\end{equation*}
	If $I_s=J_s$, we can write $\star_{J_s,I_s}$ briefly as $\star_{J_s}$.
	\item \textbf{Deconvolution}: Suppose $I_s, J_s \in \mathbb N^*$, and we use $\star^{J_s,I_s}$ to express deconvolution with strides$=(J_s, I_s)$ and the components of $\mathsf K \star^{J_s,I_s} \mathsf v$ are
	\begin{equation*}
	(\mathsf K \star^{J_s,I_s} \mathsf v)_{l,jJ_s+j^\prime,iI_s+i^\prime} =\sum_{k=1}^S \sum _{\jmath ,\imath\in \mathbb Z} \mathsf K_{l,k ,\jmath J_s + j^\prime,\imath I_s+i^\prime} \mathsf v_{\kappa , j+\jmath,i+\imath}.
	\end{equation*}
	If $I_s=J_s$, we can write $\star^{J_s,I_s}$ briefly as $\star^{J_s}$. We can also regard deconvolution as a convolution after upsampling. For example, we have $$\mathsf K \star (\mathsf v \uparrow) = \mathsf K \star^2 \mathsf v.$$
\end{enumerate}

\subsection{Representing Discretized PDEs in Convolutions}

We discuss how to transform the $\utilde {\mathcal A}_\eta$ to the convolution form. We first write components of \eqref{pde} in the following matrix form
\begin{equation}\label{PDEMatrix}
\begin{pmatrix}
\mathcal A_{1,1}&\mathcal A_{1,2}&\cdots &\mathcal A_{1, n}\\
\mathcal A_{2,1}&\mathcal A_{2,2}&\cdots &\mathcal A_{2, n}\\
\vdots&\vdots&\ddots&\vdots&\\
\mathcal A_{n,1}&\mathcal A_{n,2}&\cdots &\mathcal A_{n, n}\\
\end{pmatrix} 
\begin{pmatrix}
u^1\\ u^2\\\vdots \\ u^n
\end{pmatrix}  =
\begin{pmatrix}
f^1\\ f^2\\\vdots \\ f^n
\end{pmatrix},
\end{equation}
where each linear differential operator $\mathcal A_{i,j}(i,j=1,2,...,n)$ is a component of $\utilde {\mathcal A}_\eta$. Thanks to the linear superposition property of each component in \eqref{PDEMatrix}, we only need to consider each component separately, i.e., for given linear differential operator $\mathcal K$ and a functional $f$ in $\mathfrak V'$, find $v\in \mathfrak V$ to satisfy
\begin{equation}\label{pde:one-component}
	\mathcal K  v = f.
\end{equation}
Our goal is to find a kernel $\mathsf K$, and tensors $\mathsf v$ and $\mathsf f$ to represent the discretization of \eqref{pde:one-component} as
\begin{equation}\label{pde:one-component-conv}
	\mathsf K \star \mathsf v = \mathsf f.
\end{equation}
Suppose that the Galerkin method, e.g. the finite difference method (FDM) or finite element method (FEM), is used to discretize the PDEs \eqref{pde:one-component}. Now, we propose a general way to convert \eqref{pde:one-component} into the convolution form \eqref{pde:one-component-conv}. According to the Galerkin method, we first convert PDEs \eqref{pde:one-component} to its weak form: 
\begin{equation}\label{pde:one-component-weak}
 	\text{find } v\in \mathfrak V, \text{such that } \forall w\in \mathfrak V, K(v, w) =  f(w).
\end{equation}
Then, we choose a finite dimensional subspace $\mathfrak V_h \subset \mathfrak V$ to
 discretize \eqref{pde:one-component-weak} and solve the projected problem:
\begin{equation*}
	\text{find } v_h\in \mathfrak V_h,\  \text{such that}\ \forall w_h\in \mathfrak V_h, K(v_h, w_h) =  f(w_h).
\end{equation*}
Let $\Phi$ be a set of basis of $\mathfrak V_h$ and assume that it satisfies the following assumptions.
\begin{assumption}\label{assumption}
Suppose $\Phi$ can be divided into $S$ groups $\Phi_1,..., \Phi_S$, where $\Phi_k = \{\phi_{k,j,i}| i=0,1,2...,I_k, j=0,1,2,...,J_k\}$ and each $\Phi_k$ can be generated by translations along the grid-lines from a compact support function $\varphi_k$, i.e. $\exists\varphi_k$ $\forall i\in\{0,1,2,...,I_k\},j\in\{0,1,2,...,J_k\}$ such that
	\begin{equation*}
	\phi_{k,j,i}(x, y) =  \varphi_{k}(x-ih,y-jh).
	\end{equation*}
\end{assumption}
Then, we have the following theorem stating that the PDE \eqref{pde:one-component} discretized by FDM or FEM can be expressed in convolution form.
\begin{theorem}\label{thm:1}
	If a discretized scheme in FDM or FEM satisfies Assumption \ref{assumption}, the discretized PDE $	\mathcal K  v = f$ can be written in the form $\mathsf K \star \mathsf v = \mathsf f$ with $\mathsf K$ and $\mathsf f$ given as follows
\begin{equation}\label{Kernelvalue}
\mathsf K =	(\mathsf K_{l,k,j,i}) = K(\varphi_{k}(x-ih,y-jh), \varphi_{l}(x, y))\quad\mbox{and}\quad\mathsf f=(\mathsf f_{l,j,i}) = f(\phi_{l,j,i}).
\end{equation}
\end{theorem}
\begin{proof}
	With Assumption \ref{assumption}, we can write $v_h$ as
	\begin{equation*}
	v_h = \sum_{k=1}^S\sum_{\jmath=0}^{J_k}\sum_{\imath=0}^{I_k} \mathsf v_{k,\jmath,\imath}\phi_{k,\jmath,\imath}.
	\end{equation*}
	According to Galerkin method, we have
	\begin{equation*}
	\begin{aligned}
	K(v_h, \phi_{l,j,i}) &= K(\sum_{k=1}^S\sum_{\jmath=0}^{J_k}\sum_{\imath=0}^{I_k} \mathsf v_{k,\jmath,\imath}\phi_{k,\jmath,\imath}, \phi_{l,j,i})=\sum_{k=1}^S\sum_{\jmath=0}^{J_k}\sum_{\imath=0}^{I_k}\mathsf v_{k,\jmath,\imath}K(\phi_{k,\jmath,\imath}, \phi_{l,j,i})\\
	&=\sum_{k=1}^S\sum_{\jmath=0}^{J_k}\sum_{\imath=0}^{I_k}\mathsf v_{k,\jmath,\imath}K(\varphi_{k}(x-\imath h, y -\jmath h), \varphi_{l}(x-ih, y-jh))\\
	&=\sum_{k=1}^S\sum_{\jmath=-j}^{J_k-j}\sum_{\imath=-i}^{I_k-i}\mathsf v_{k,\jmath + j,\imath + i}K(\varphi_{k}(x-(\imath+i) h, y -(\jmath+j) h), \varphi_{l}(x-ih, y-jh))\\
	&=\sum_{k=1}^S\sum_{\jmath=-j}^{J_k-j}\sum_{\imath=-i}^{I_k-i}\mathsf v_{k,\jmath + j,\imath + i}K(\varphi_{k}(x-\imath h, y -\jmath h), \varphi_{l}(x, y)).
	\end{aligned}
	\end{equation*}
	Let $\mathsf K = (\mathsf K_{l,k,j,i})$, where
	\begin{equation}
	\mathsf K_{l,k,j,i} = K(\varphi_{k}(x-ih,y-jh), \varphi_{l}(x, y)).
	\end{equation}
	While $i\notin \{0,1,...,I_k\}$ or $j\notin \{0,1,...,J_k\}$, we set $\mathsf v_{k,j,i} = 0, \mathsf K_{l,k,j,i}=0$.
	Then, we have 
	\begin{equation*}
	K(v_h, \phi_{l,j,i})=\sum_{k=1}^S\sum_{\jmath=-j}^{J_k-j}\sum_{\imath=-i}^{I_k-i}\mathsf v_{k,\jmath + j,\imath + i}\mathsf K_{l,k,\jmath, \imath}=\sum_{k=1}^S\sum_{\jmath\in \mathbb Z}\sum_{\imath\in \mathbb Z}\mathsf v_{k,\jmath + j,\imath + i}\mathsf K_{l,k,\jmath,\imath}=(\mathsf K \star \mathsf v)_{l,j,i}.
	\end{equation*}
	Let
	$\mathsf f = (\mathsf f_{l,j,i})$ with $\mathsf f_{l,j,i} = f(\phi_{l,j,i})$.
	We obtain
	$$K(v_h, \phi_{l,j,i})= (f(\phi_{l,j,i}),\quad \forall l=1,2,...,S,$$  which is the same as \eqref{pde:one-component-conv}.
\end{proof}

Although we have demonstrated a generic method to convert PDEs \eqref{pde:one-component} into the convolution form, it is sometimes inconvenient to use. We can calculate $\mathsf K$ in an easier way for most discretization schemes, as will be described in the remaining part of this subsection. For both FDM and FEM discretization, we assume that the mesh $\mathcal T$ is $N\times N$ uniform triangular or rectangular mesh, and let $h = \dfrac 1N$.

\subsubsection{Finite Difference Methods (FDM)}
FDM is one of the most popular discretizations used for solving PDEs by approximating them with difference equations. The basis functions of FDM are Legendre polynomials. Thus, it is not convenient to use \eqref{Kernelvalue} to calculate $\mathsf K$ and $\mathsf f$. Fortunately, we can use Taylor's expansion to compute entries of $\mathsf K$ and $\mathsf f$.
Furthermore, the functions $v$ and $f$ can be easily discretized by using their values restricted on the mesh $\mathcal T$. As examples, we present four commonly seen cases as follows (see \textbf{Figure \ref{FDMpoint}}).
\begin{enumerate}[(1)]
	\item Vertex of an element: set $\mathsf v_{j,i} = v(ih,jh),\ i,j = 0,1,...,N$;
	\item Midpoint of a horizontal edge: set $\mathsf v_{j,i} = v((i-0.5)h,jh),\ i = 1,2,...,N, j = 0,1,...,N$;
	\item Midpoint of a vertical edge: set $\mathsf v_{j,i} = v(ih,(j-0.5)h),\ i = 0,1,...,N, j = 1,2,...,N$;
	\item Center of an element, set $\mathsf v_{j,i} = v((i-0.5)h,(j-0.5)h),\ i,j = 1,2,...,N$.
\end{enumerate}
\begin{figure}
	\flushleft
	\hspace{2.7cm}
\begin{tikzpicture}[thick,scale=2]
\coordinate (A1) at (0, 0);
\coordinate (A2) at (0, 1);
\coordinate (A3) at (1, 1);
\coordinate (A4) at (1, 0);

\fill (A1) circle (1pt);
\fill (A2) circle (1pt);
\fill (A3) circle (1pt);
\fill (A4) circle (1pt);

\draw[thick] (A1) -- (A2);
\draw[thick] (A2) -- (A3);
\draw[thick] (A3) -- (A4);
\draw[thick] (A4) -- (A1);
\draw[fill=black!20,opacity=0.5] (A1) -- (A2) -- (A3) -- (A4);
\end{tikzpicture}
\quad\quad 
\begin{tikzpicture}[thick,scale=2]
\coordinate (A1) at (0, 0);
\coordinate (A2) at (0, 1);
\coordinate (A3) at (1, 1);
\coordinate (A4) at (1, 0);

\fill (0.5, 0) circle (1pt);
\fill (0.5, 1) circle (1pt);

\draw[thick] (A1) -- (A2);
\draw[thick] (A2) -- (A3);
\draw[thick] (A3) -- (A4);
\draw[thick] (A4) -- (A1);
\draw[fill=black!20,opacity=0.5] (A1) -- (A2) -- (A3) -- (A4);
\end{tikzpicture}
\quad\quad
\begin{tikzpicture}[thick,scale=2]
\coordinate (A1) at (0, 0);
\coordinate (A2) at (0, 1);
\coordinate (A3) at (1, 1);
\coordinate (A4) at (1, 0);

\fill (0, 0.5) circle (1pt);
\fill (1, 0.5) circle (1pt);

\draw[thick] (A1) -- (A2);
\draw[thick] (A2) -- (A3);
\draw[thick] (A3) -- (A4);
\draw[thick] (A4) -- (A1);
\draw[fill=black!20,opacity=0.5] (A1) -- (A2) -- (A3) -- (A4);
\end{tikzpicture}
\quad\quad
\begin{tikzpicture}[thick,scale=2]
\coordinate (A1) at (0, 0);
\coordinate (A2) at (0, 1);
\coordinate (A3) at (1, 1);
\coordinate (A4) at (1, 0);

\fill (0.5, 0.5) circle (1pt);
\fill (0.5, 0.5) circle (1pt);

\draw[thick] (A1) -- (A2);
\draw[thick] (A2) -- (A3);
\draw[thick] (A3) -- (A4);
\draw[thick] (A4) -- (A1);
\draw[fill=black!20,opacity=0.5] (A1) -- (A2) -- (A3) -- (A4);
\end{tikzpicture}
$$(1) \hspace{2.5cm} (2) \hspace{2.5cm} (3) \hspace{2.5cm} (4)$$
\caption{Different case of point-wise discretion.}\label{FDMpoint}
\end{figure}
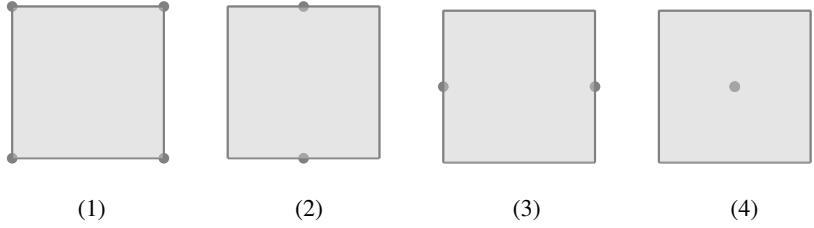

After choosing a discretization for $v$ and $f$, we use Taylor expansion to discretize $\mathcal K$. Suppose $p\in \mathbb N, \bm \alpha = (\alpha_1, \alpha_2), |\bm \alpha |= \alpha_1+\alpha_2, \partial^{\bm \alpha} = \dfrac{\partial^{\alpha_1+\alpha_2}}{(\partial x)^{\alpha_1}(\partial y)^{\alpha_2}}$ and 
$\mathcal{K} = \sum_{|\bm \alpha|\leqslant p} a_{\bm \alpha} \partial^{\bm \alpha}$.
Then, we can obtain the kernel $\mathsf K$ from a difference scheme. If $\forall \bm\alpha$ such that $\bm\alpha\leqslant p$, we have a finite difference approximation of $\partial^{\bm \alpha}$ as
$
\partial^{\bm \alpha}_h \mathsf v_{j,i} \approx \sum_{\jmath,\imath} q^{\bm \alpha}_{\jmath,\imath}\mathsf v_{j+\jmath,i+\imath}
$
and the expression of each component of $\mathsf K$ is
$
\mathsf K_{\jmath, \imath} = \sum_{|\bm \alpha|\leqslant p} \sum_{\jmath,\imath} a_{\bm\alpha}q^{\bm\alpha}_{\jmath, \imath}.
$
Common finite difference approximations of partial derivatives and their corresponding convolution kernels are listed in \textbf{Table \ref{FDMtable}}.
\begin{table}[width=.9\linewidth,cols=3,pos=h]
	\begin{tabular}{|c|c|c|}
		\hline
		$\mathcal K$&difference scheme&kernel $\mathsf K$\\\hline
		\multirow{2}{*}{$\partial_x$}& $\frac{1}{h}(\mathsf v_{i,j}-\mathsf v_{i-1,j})$&$\frac{1}{h}\begin{pmatrix}
		-1& 1& 0
		\end{pmatrix}$\\\cline{2-3}
		&$\frac{1}{h}(\mathsf v_{i+1,j}-\mathsf v_{i,j})$& $\frac{1}{h}\begin{pmatrix}
		0 &-1& 1
		\end{pmatrix}$\\\hline
		\multirow{2}{*}{$\partial_y$}&$\frac{1}{h}(\mathsf v_{i,j}-\mathsf v_{i,j-1})$&
		$\dfrac{1}{h}\begin{pmatrix}
		-1\\ 1\\ 0
		\end{pmatrix}$\\\cline{2-3}
		&$\frac{1}{h}(\mathsf v_{i,j+1}-\mathsf v_{i,j})$&$\dfrac{1}{h}\begin{pmatrix}
		0 \\-1\\ 1
		\end{pmatrix}$\\\hline
		$\partial_{xx}$&$\frac{1}{h^2}(\mathsf v_{i-1,j}+\mathsf v_{i+1,j}-2\mathsf v_{i,j})$ &$\frac{1}{h^2}\begin{pmatrix}
		1& -2& 1
		\end{pmatrix}$\\\hline
		$\partial_{xy}$ &$\frac{1}{4h^2}(\mathsf v_{i-1,j-1}+\mathsf v_{i+1,j+1}-\mathsf v_{i+1,j-1}- \mathsf v_{i-1,j+1})$&$\dfrac{1}{4h^2}\begin{pmatrix}
		1& 0& -1\\
		0& 0& 0\\
		-1& 0& 1\\
		\end{pmatrix}$\\\hline
		$\partial_{yy}$&$\frac{1}{h^2}(\mathsf v_{i,j-1}+ \mathsf v_{i,j+1}-2\mathsf v_{i,j})$&$\dfrac{1}{h^2}\begin{pmatrix}
		1\\ -2\\ 1
		\end{pmatrix}$ \\\hline
		$\Delta$&$\frac{1}{h^2}(\mathsf v_{i-1,j}+\mathsf v_{i+1,j}+\mathsf v_{i,j-1}+ \mathsf v_{i,j+1}-4\mathsf v_{i,j})$ &$\dfrac{1}{h^2}\begin{pmatrix}
		0& 1 &0\\1& -4& 1\\0& 1& 0
		\end{pmatrix}$\\\hline
	\end{tabular}
\caption{Some common finite difference schemes and its corresponding kernels $\mathsf K$.}\label{FDMtable}
\end{table}

\subsubsection{Finite Element Methods (FEM)}

FEM is more complicated than FDM because a variable $v$ usually contains several types of basis functions. Each type of basis functions is a channel of the tensor $\mathsf v$. For most FEM methods on the rectangle mesh, we can divide each basis functions into the following 4 cases based on the support of the functions, i.e. $2\times 2$, $2\times 1$, $1\times 2$, and $1\times 1$ elements (see \textbf{Figure \ref{FEMrectsupport}}).
These four cases can be reduced to the case for FDM.
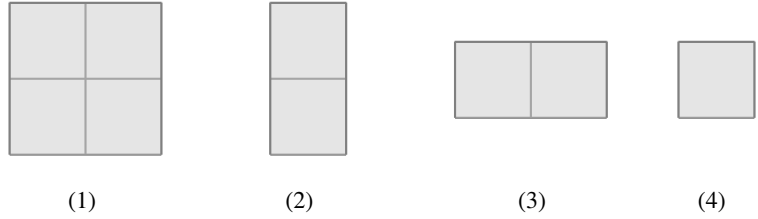
\begin{figure}[pos=htbp]
	\flushleft
	\hspace{3cm} 
	\begin{tikzpicture}[thick,scale=2]
	\coordinate (A1) at (0, 0);
	\coordinate (A2) at (0, 1);
	\coordinate (A3) at (1, 1);
	\coordinate (A4) at (1, 0);
	
	\draw[thick] (A1) -- (A2);
	\draw[thick] (A2) -- (A3);
	\draw[thick] (A3) -- (A4);
	\draw[thick] (A4) -- (A1);
	\draw[thick] (0,0.5) -- (1,0.5);
	\draw[thick] (0.5, 0) -- (0.5, 1);
	\draw[fill=black!20,opacity=0.5] (A1) -- (A2) -- (A3) -- (A4);
	\end{tikzpicture}
	\quad\quad\quad\quad 
	\begin{tikzpicture}[thick,scale=2]
	\coordinate (A1) at (0.25, 0);
	\coordinate (A2) at (0.25, 1);
	\coordinate (A3) at (0.75, 1);
	\coordinate (A4) at (0.75, 0);
	
	\draw[thick] (0.25,0.5) -- (0.75,0.5);
	
	\draw[thick] (A1) -- (A2);
	\draw[thick] (A2) -- (A3);
	\draw[thick] (A3) -- (A4);
	\draw[thick] (A4) -- (A1);
	\draw[fill=black!20,opacity=0.5] (A1) -- (A2) -- (A3) -- (A4);
	\end{tikzpicture}
	\quad\quad\quad\quad 
	\begin{tikzpicture}[thick,scale=2]
	\coordinate (A1) at (0, 0.25);
	\coordinate (A2) at (0, 0.75);
	\coordinate (A3) at (1, 0.75);
	\coordinate (A4) at (1, 0.25);
	\fill (0, 0) circle (0pt);
	\draw[thick] (A1) -- (A2);
	\draw[thick] (A2) -- (A3);
	\draw[thick] (A3) -- (A4);
	\draw[thick] (A4) -- (A1);
	\draw[thick] (0.5,0.25) -- (0.5,0.75);
	\draw[fill=black!20,opacity=0.5] (A1) -- (A2) -- (A3) -- (A4);
	\end{tikzpicture}
	\quad
	\begin{tikzpicture}[thick,scale=2]
	\coordinate (A1) at (0.25, 0.25);
	\coordinate (A2) at (0.25, 0.75);
	\coordinate (A3) at (0.75, 0.75);
	\coordinate (A4) at (0.75, 0.25);
		\fill (0, 0) circle (0pt);
	
	\draw[thick] (A1) -- (A2);
	\draw[thick] (A2) -- (A3);
	\draw[thick] (A3) -- (A4);
	\draw[thick] (A4) -- (A1);
	\draw[fill=black!20,opacity=0.5] (A1) -- (A2) -- (A3) -- (A4);
	\end{tikzpicture}
	$$\hspace{3cm} (1) \hspace{2.5cm} (2) \hspace{2.7cm} (3) \hspace{2cm} (4) \hspace{3cm}$$
	\caption{Some usual support of base function in FEM on rectangular mesh.}\label{FEMrectsupport}
\end{figure}

On the triangular mesh, each basis function can be divided into 6 cases, which are shown in \textbf{Figure \ref{FEMtrisupport}} according to the shape of the support of the function. For case (1)$\sim$(4), they can also be reduced to the case of FDM, while for case (5) or (6), we may have to apply \textbf {Theorem \ref{thm:1}} to calculate $\mathsf K$ and $\mathsf f$.
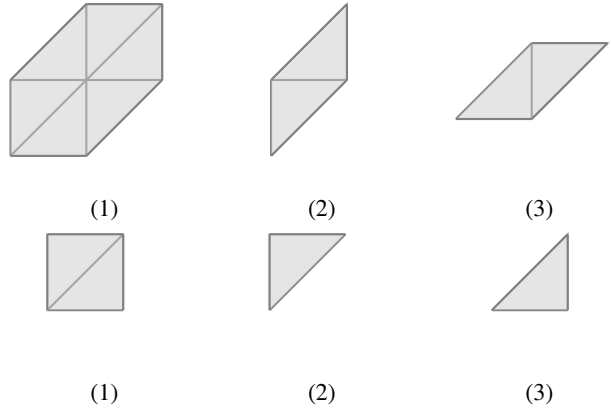
\begin{figure}[pos=htbp]
	\flushleft
	\hspace{4cm} 
	\begin{tikzpicture}[thick,scale=2]
	\coordinate (A1) at (0, 0);
	\coordinate (A2) at (0.5, 0);
	\coordinate (A3) at (1, 0.5);
	\coordinate (A4) at (1, 1);
	\coordinate (A5) at (0.5, 1);
	\coordinate (A6) at (0, 0.5);
	
	\draw[thick] (A1) -- (A2);
	\draw[thick] (A2) -- (A3);
	\draw[thick] (A3) -- (A4);
	\draw[thick] (A4) -- (A5);
	\draw[thick] (A5) -- (A6);
	\draw[thick] (A6) -- (A1);
	\draw[thick] (0,0) -- (1, 1);
	\draw[thick] (0,0.5) -- (1,0.5);
	\draw[thick] (0.5, 0) -- (0.5, 1);
	\draw[fill=black!20,opacity=0.5] (A1) -- (A2) -- (A3) -- (A4)--(A5)--(A6);
	\end{tikzpicture}
	\quad\quad\quad\quad 
	\begin{tikzpicture}[thick,scale=2]
	\coordinate (A1) at (0.25, 0);
	\coordinate (A2) at (0.75, 0.5);
	\coordinate (A3) at (0.75, 1);
	\coordinate (A4) at (0.25, 0.5);
	
	\draw[thick] (0.25,0.5) -- (0.75,0.5);
	
	\draw[thick] (A1) -- (A2);
	\draw[thick] (A2) -- (A3);
	\draw[thick] (A3) -- (A4);
	\draw[thick] (A4) -- (A1);
	\draw[fill=black!20,opacity=0.5] (A1) -- (A2) -- (A3) -- (A4);
	\end{tikzpicture}
	\quad\quad\quad\quad 
	\begin{tikzpicture}[thick,scale=2]
	\coordinate (A1) at (0, 0.25);
	\coordinate (A2) at (0.5, 0.25);
	\coordinate (A3) at (1, 0.75);
	\coordinate (A4) at (0.5, 0.75);
	\fill (0, 0) circle (0pt);
	\draw[thick] (A1) -- (A2);
	\draw[thick] (A2) -- (A3);
	\draw[thick] (A3) -- (A4);
	\draw[thick] (A4) -- (A1);
	\draw[thick] (0.5,0.25) -- (0.5,0.75);
	\draw[fill=black!20,opacity=0.5] (A1) -- (A2) -- (A3) -- (A4);
	\end{tikzpicture}
	
	$$\hspace{4cm} (1) \hspace{2.5cm} (2) \hspace{2.5cm} (3) \hspace{4cm}$$
	
	\hspace{4cm}
	\begin{tikzpicture}[thick,scale=2]
	\coordinate (A1) at (0.25, 0.25);
	\coordinate (A2) at (0.25, 0.75);
	\coordinate (A3) at (0.75, 0.75);
	\coordinate (A4) at (0.75, 0.25);
	\fill (0, 0) circle (0pt);
	
	\draw[thick] (A1) -- (A2);
	\draw[thick] (A2) -- (A3);
	\draw[thick] (A3) -- (A4);
	\draw[thick] (A4) -- (A1);
	\draw[thick] (0.25,0.25) -- (0.75,0.75);
	\draw[fill=black!20,opacity=0.5] (A1) -- (A2) -- (A3) -- (A4);
	\end{tikzpicture}
	\hspace{1.2cm}
	\begin{tikzpicture}[thick,scale=2]
	\coordinate (A1) at (0.25, 0.25);
	\coordinate (A2) at (0.25, 0.75);
	\coordinate (A3) at (0.75, 0.75);
	\fill (0, 0) circle (0pt);
	
	\draw[thick] (A1) -- (A2);
	\draw[thick] (A2) -- (A3);
	\draw[thick] (A3) -- (A1);
	\draw[fill=black!20,opacity=0.5] (A1) -- (A2) -- (A3);
	\end{tikzpicture}
		\hspace{1.2cm}
	\begin{tikzpicture}[thick,scale=2]
	\coordinate (A1) at (0.25, 0.25);
	\coordinate (A3) at (0.75, 0.75);
	\coordinate (A4) at (0.75, 0.25);
	\fill (0, 0) circle (0pt);
	
	\draw[thick] (A1) -- (A3);
	\draw[thick] (A3) -- (A4);
	\draw[thick] (A4) -- (A1);
	\draw[fill=black!20,opacity=0.5] (A1) -- (A3) -- (A4);
	\end{tikzpicture}
	$$\hspace{4cm} (1) \hspace{2.5cm} (2) \hspace{2.5cm} (3) \hspace{4cm}$$
	\caption{Some usual support of base function in FEM on triangular mesh.}\label{FEMtrisupport}
\end{figure}

\section{The Multigrid Method and PDE-MgNet}

In this section, we briefly describe the geometric MG method and PDE-MgNet. MG method is one of the most high-efficiency methods for solving PDEs. We consider the discrete form of the parameterized PDEs \eqref{pde:discrete}.

\subsection{ Multigrid Method}

Iterative method is one of the basic numerical methods for solving the linear system \eqref{pde:discrete}. Given an initial guess $\mathbf u_0$ and an update scheme represented by $\mathcal H$, we can write the iterative method generically as
\begin{equation}\label{iteration}
\mathbf u_{t+1} = \mathbf u_t +\mathcal H (\mathbf u_t),\  t=0,1\ldots,T.
\end{equation}
Note that we can regard \eqref{iteration} as a dynamical system or a feed-forward network. Such perspective is the key to connect deep neural networks (e.g. ResNet \citep{he2016deep}) with dynamic systems, PDEs and optimal control \citep{gregor2010learning,Chen_2015,e2017a,haber2017stable,lu2018beyond,chen2018neural}. 

In addition to the iterative scheme \eqref{iteration}, the residual correction scheme 
\begin{equation}\label{iterationR}
\mathbf u_{t+1} = \mathbf u_t +\mathcal H (\mathbf f - \mathbf A\mathbf u_t), ,\  t=0,1\ldots,T,
\end{equation}
is one of most important types of iterative method. The MG method is one of such iterative methods, which is written as:
 \begin{equation}\label{iterationMg}
 \mathbf u_{t+1} = \mathbf u_t +\mathrm {Mg} (\mathbf f- \mathbf A\mathbf u_t),\  t=0,1\ldots,T,
 \end{equation}
where the Mg operator in \eqref{iterationMg} is given by \textbf{Algorithm \ref{alg:MG-Slash}}.

The MG operator can be divide into two steps, i.e. smoothing and coarse grid correction. The smoothing step is to use a smoother to eliminate high-frequency errors. A smoother $\mathcal B$ (or its matrix form $\mathbf B$) is often an iterative scheme by itself taking the form
\begin{equation*}
\mathbf u_+ = \mathbf u_0 +\mathbf B (\mathbf f - \mathbf A \mathbf u_0).
\end{equation*} 
Popular choices of the smoother in MG include the Jacobi and Gauss-Seidel (GS) smoother, which are listed below
\begin{equation}\label{BaseSmoother}
\mathbf B=\begin{cases}
\text{diag}(\mathbf A)^{-1} & \text{Jacobi }\\
\text{tril}(\mathbf A)^{-1}&\text{GS}
\end{cases}.
\end{equation}

The Jacobi or GS smoother can efficiently eliminate high-frequency approximation errors. However, they are ineffective for low-frequency errors, which is where the coarse grid correction is needed. The MG methods utilize the solution on a coarse grid to approximate the low-frequency error. As a simple example, we illustrate the steps of coarse grid correction in a two-level MG in \textbf{Figure \ref{fig:mg}}.
\begin{figure}[pos=htbp]
	\centering
	\includegraphics[width=0.6\linewidth]{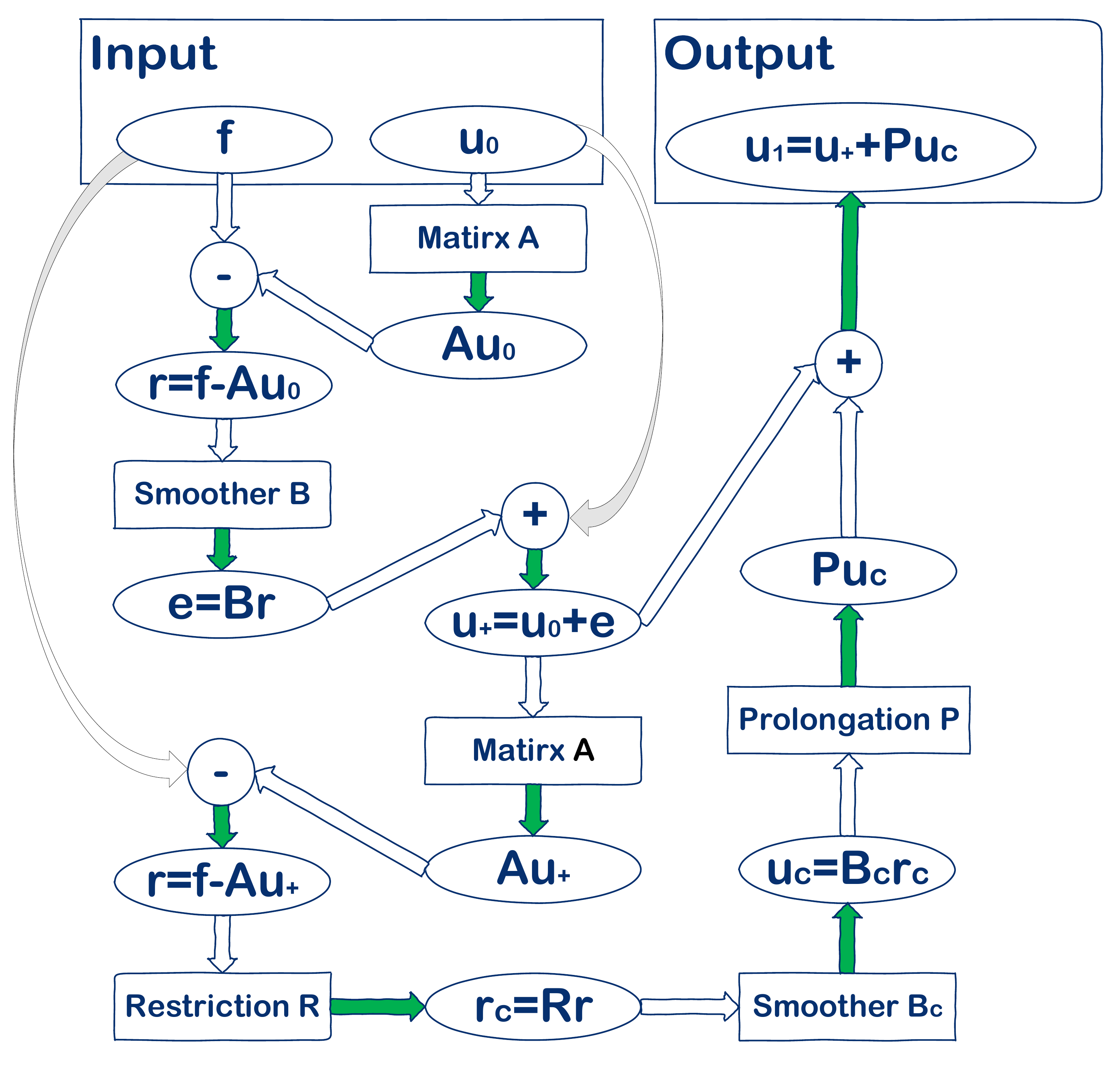}
	\caption{Two level $\backslash$-Cycle MG.}
	\label{fig:mg}
\end{figure}
For the multi-level MG method, if we have a sequence of grids $\mathcal T^1(=\mathcal T), \mathcal T^2,..., \mathcal T^{J}$ and assume that the prolongation and restriction between $\mathcal T^\ell$ and $\mathcal T^{\ell+1}$ are $\mathbf P^\ell_{\ell+1}$ and $\mathbf R^{\ell+1}_{\ell}$. Let $\mathbf A^{\ell+1} = \mathbf R^{\ell+1}_{\ell} \mathbf A^\ell \mathbf P^\ell_{\ell+1}$. The \textbf{Algorithm \ref{alg:MG-Slash}} presents the algorithm of multi-level MG.

\begin{algorithm}[htb]
	\caption{$\mathbf u = {\text{Mg}}(\mathbf f; J,\nu_1, \cdots, \nu_J)$}
	\label{alg:MG-Slash}
	\begin{algorithmic}
		\State Hyper-parameters: number of grids $J$,
		times of smooth in each grid: $\nu_1, \cdots, \nu_J$ 
		\State Input: right hand side $\mathbf f$
		\State Output: approximate solution $\mathbf u$
		\State Initialization:
		$$
		\mathbf f^1 \leftarrow \mathbf f, \quad \mathbf u^{1,0} \leftarrow\mathbf 0, \quad \mathbf r^{1,0} \leftarrow \mathbf f.
		$$
		\State Smoothing and restriction from fine to coarse level (nested)
		\For{$\ell = 1:J$}
		\State Smoothing
		\If{$\ell = J$}
		$$\mathbf u^{\ell,1} \leftarrow (\mathbf A^\ell)^{-1}\mathbf r^{\ell,0}$$
		\Else \For{$i = 1:\nu_\ell$}
		$$\mathbf u^{\ell,i} \leftarrow \mathbf u^{\ell,i-1} +\mathbf B^\ell\mathbf r^{\ell,i-1}$$
		$$\mathbf r^{\ell,i} \leftarrow \mathbf f^\ell - \mathbf  A^\ell\mathbf u^{\ell,i}. $$
	
		\EndFor
		\EndIf
		\State Form restricted residual
		$$
		\mathbf f^{\ell+1} \leftarrow \mathbf R^{\ell+1}_\ell\mathbf r^{\ell,\nu_\ell},
		\quad \mathbf u^{\ell+1,0} \leftarrow \mathbf 0,
		\quad \mathbf r^{\ell+1,0}\leftarrow \mathbf f^{\ell+1}.
		$$
		\EndFor
		\State Prolongation and restriction from coarse to fine level
		\For{$\ell = J-1:1$}
		\State Coarse grid correction
		$$
		\mathbf u^{\ell,\nu_\ell} \leftarrow\mathbf u^{\ell,\nu_\ell} + \mathbf P_{\ell+1}^{\ell}\mathbf u^{\ell+1, \nu_{\ell}}.
		$$
		\EndFor
		
		\State \Return	$\mathbf u = \mathbf u^{1,\nu_1}$.
	\end{algorithmic}
\end{algorithm}

\subsection{PDE-MgNet}

The original MgNet proposed by \citep{MgNet} is a new CNN model for image classification, which is inspired by the connection between CNNs and the MG method. As observed by \citep{MgNet} that the smoothers, prolongations, and restrictions can be represented by convolutions. Thus, we can write the corresponding kernels as $\mathsf B$, $\mathsf P$, and $\mathsf R$. Therefore, the MG method can be naturally reformulated as a CNN model, which was called the MgNet by \citep{MgNet}. 

In this paper, we focus on solving PDEs. Thus, we need to modify the original MgNet in \citep{MgNet} to be suitable for solving PDEs. We shall call the modified MgNet the PDE-MgNet. We formulate every components of PDE-MgNet in the convolution form. For that, we consider the convolution form of PDEs \eqref{pde}:
\begin{equation}\label{pde:discrete_conv}
\mathsf A_\eta \star \mathsf u =\mathsf f.
\end{equation}
PDE-MgNet replaces the smoother $\mathcal B^\ell$, the prolongation $\mathcal P$, the restriction $\mathcal R$ and $\mathbf A^\ell$ in the MG methods described in \textbf{Algorithm \ref{alg:MG-Slash}} with trainable convolution operators. \textbf{Figure} \ref{fig:mgnet} shows the architecture of a two-level $\backslash$-Cycle PDE-MgNet and the multi-level case of PDE-MgNet is presented in \textbf{Algorithm \ref{alg:L-Slash0}}.
\begin{figure}[pos=htbp]
	\centering
	\includegraphics[width=0.6\linewidth]{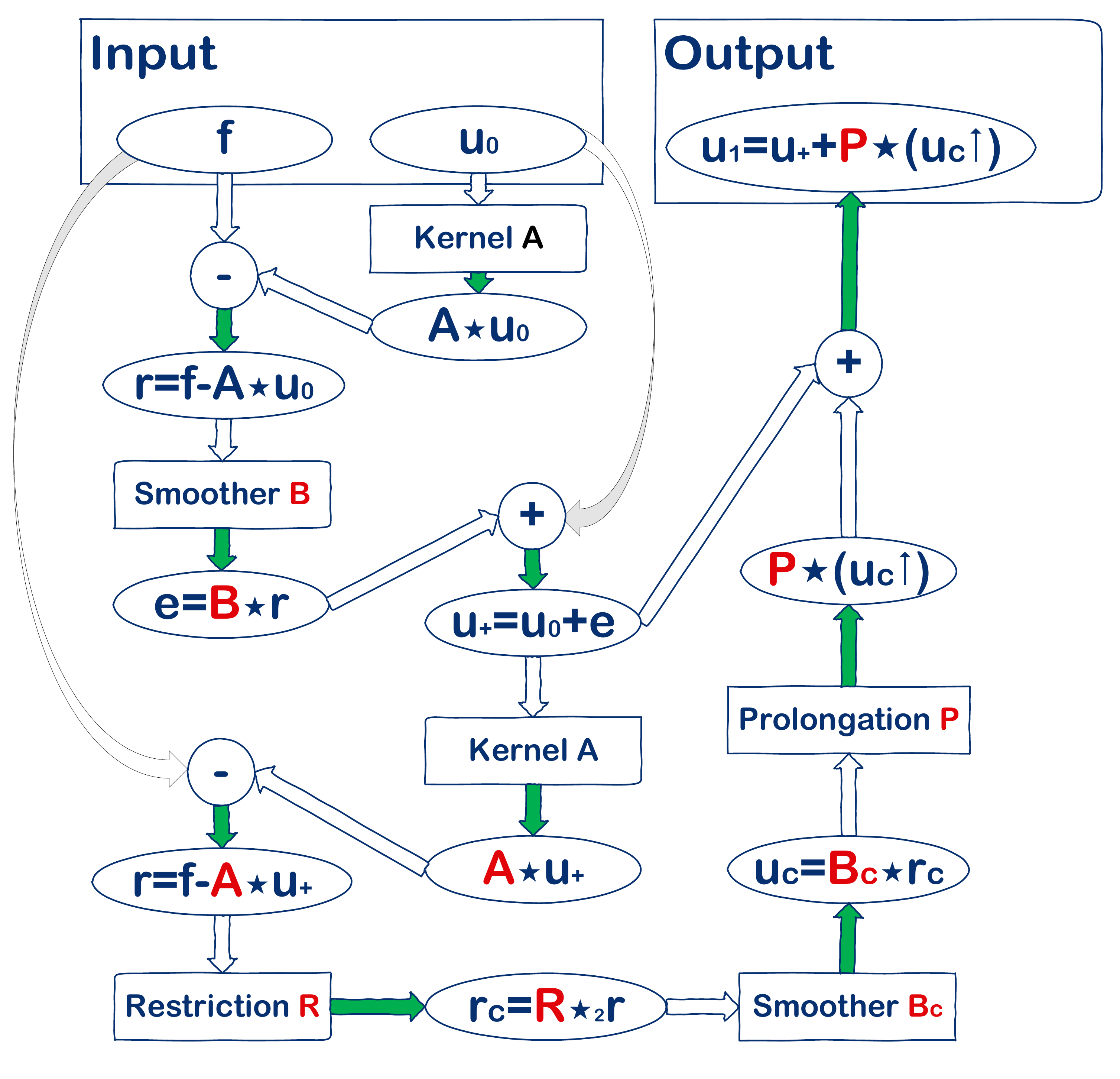}
	\caption{Two level $\backslash$-Cycle PDE-MgNet.}
	\label{fig:mgnet}
\end{figure}

\begin{algorithm}[!htp]
	\caption{$\mathsf u = {\text{PDE-MgNet}}(\mathsf f; J,\nu_1, \cdots, \nu_J)$}
	\label{alg:L-Slash0}
	\begin{algorithmic}
		\State Hyper-parameters: number of grids $J$,
		times of smooth in each grid: $\nu_1, \cdots, \nu_J$ 
		\State Input: right-hand side$\mathsf f$
		\State Output: approximate solution $\mathsf u$
		\State Initialization
		$$
		\mathsf f^1 \leftarrow\mathsf f, 
		\quad\mathsf u^{1,0}\leftarrow\mathsf 0,
		\quad \mathsf r^{1,0} \leftarrow\mathsf f.
		$$
		\State Smoothing and restriction from fine to coarse level
		\For{$\ell = 1:J$}
		
		\State Smoothing:
		\If{$\ell = J$}
		
		Convert $\mathsf r^{\ell, 0}$ into vector form $\mathbf r^{\ell,0}$ and $\mathsf A^{\ell}$ into matrix form $\mathbf A^{\ell}$.
		$$\mathbf u^{\ell,1} \leftarrow (\mathbf A^\ell)^{-1}\mathbf r^{\ell,0}.$$
		
		Convert $\mathbf u^{\ell, 1}$ into tensor form $\mathsf u^{\ell,1}.$
		\Else \For{$i = 1:\nu_\ell$}
		\begin{equation}\label{alg:smoothing}
		  \hspace{5cm}  \mathsf u^{\ell,i} \leftarrow\mathsf u^{\ell,i-1} +\mathsf B^{\ell, i-1} \star \mathsf r^{\ell,i-1},
		\end{equation}
		$$\mathsf r^{\ell,i} \leftarrow \mathsf f^{\ell}- \mathsf A^\ell\mathsf \star\mathsf  u^{\ell,i}.$$
		\EndFor
		\EndIf
		\State Form restricted residual
		$$
		\mathsf  f^{\ell+1} \leftarrow\mathsf  R^{\ell+1}_\ell\star_2\mathsf  r^{\ell,\nu_\ell},
		\quad \mathsf u^{\ell+1,0} \leftarrow 0,
		\quad \mathsf f^{\ell+1,0} \leftarrow \mathsf  f^{\ell+1}.
		$$
		\EndFor
		\State Prolongation from coarse to fine level
		\For{$\ell = J-1:1$}
		\State Coarse grid correction
		$$\mathsf u^{\ell,\nu_\ell} \leftarrow \mathsf u^{\ell,\nu_\ell} + \mathsf P_{\ell+1}^{\ell}\star^2\mathsf u^{\ell+1, \nu_{\ell}}.$$
		\EndFor
		
		\State \Return	$\mathsf u =\mathsf  u^{1,\nu_1}$.
	\end{algorithmic}
\end{algorithm}

With \textbf{Algorithm \ref{alg:L-Slash0}}, we obtain the iterative PDE-MgNet for solving \eqref{pde:discrete_conv}:
\begin{equation}\label{iterationmg}
\mathsf u_{t+1} = \mathsf u_t + \text{PDE-MgNet}(\mathsf f -\mathsf A \star\mathsf u_t),\  t=0,1\ldots,T,
\end{equation}
with $\mathsf u_0=\mathsf 0$. For convenience, we shall refer to the iterative PDE-MgNet \eqref{iterationmg} simply as the PDE-MgNet.  Note that the PDE-MgNet \eqref{iterationmg} is precisely the MG method \eqref{iterationMg} with $\mathcal A$, $\mathcal B$, $\mathcal P$ and $\mathcal R$ replaced by convolutions, and it reduces to the original MgNet proposed by \citep{MgNet} when the prolongation step is replaced by a classifier.

The values of the convolution kernels of PDE-MgNet are learned from data by minimizing a loss function defined similarly as in \citep{2019arXiv190502789L}. Suppose the distribution of the right-and-side function $\mathsf f$ is $F$. If we sample the distribution $F$ by $M_{\text{train}}$ times and denote the training data set as $\mathfrak X_{F}$, then the empirical loss is given by
\begin{equation}\label{Loss2}
Loss	\approx \dfrac 1{M_{\text{train}}} \sum_{\mathsf f\in \mathfrak X_{F}} \dfrac {||\mathsf f -\mathsf A \star \mathsf u_T||^2}{||\mathsf f||^2}.
\end{equation}
In our experiments, we choose $T=1$ following \citep{katrutsa2017deep}.

\section{Meta-MgNet}
The PDE-MgNet suffers from poor generalization when tested on a data set generated from the distribution far away from that of the training set, which significantly limits the practicality and utility of PDE-MgNet. This motivates us to improve PDE-MgNet with meta-learning by introducing a properly designed hypernetwork which infers specific components of the PDE-MgNet according to the parameters $\eta$ of the parameterized PDE to achieve uniformly fast convergence. In this paper, the hypernetwork is introduced to make the smoother $\mathcal B$ in PDE-MgNet PDE-dependent. Now, we shall describe details of the design of such hypernetwork and the architecture of the entire Meta-MgNet.

\subsection{Architecture of Meta-MgNet}

The hypernetwork we introduce to the PDE-MgNet is called Meta-NN. The Meta-MgNet uses Meta-NN to infer an appropriate smoother (called meta-smoother) for each parameter $\eta$. The architecture of the meta-smoother in comparison with the smoother of the PDE-MgNet is presented in \textbf{Figure \ref{ComparisonPDE_Meta}}. The advantage of Meta-MgNet over PDE-MgNet is that the smoother of Meta-MgNet changes according to $\mathsf A_\eta$ and $\mathsf r$, or we can write $\mathcal B=\mathcal B_{\mathsf A_\eta, \mathsf r}$ which is realized by the Meta-NN.
\begin{figure}[pos=htbp]
	\centering
	\includegraphics[width=0.7\linewidth]{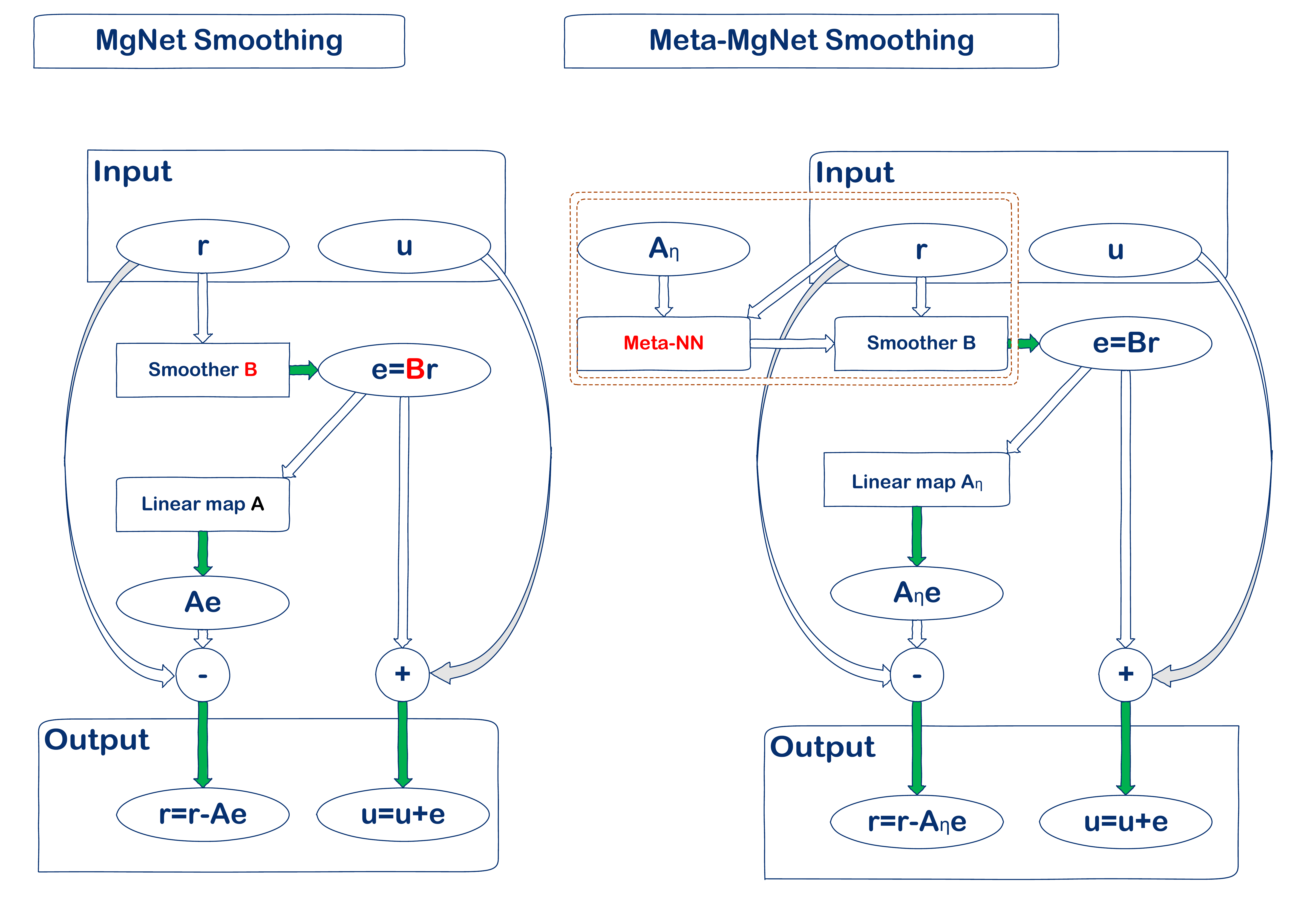}
	\caption{Comparison between PDE-MgNet and Meta-MgNet in smoothing step.  The red rectangle shows the major difference.}
	\label{ComparisonPDE_Meta}
\end{figure}
The entire architecture of the (iterative) Meta-MgNet solving \eqref{pde:discrete_conv} is given by
\begin{equation}\label{Iteration}
\mathsf u_{t+1} = \mathsf u_t + \text{Meta-MgNet}(\mathsf f -\mathsf A_\eta \star\mathsf u_t,\mathsf A_\eta ),
\end{equation}
where the Meta-MgNet$(\cdot)$ is computed using \textbf{Algorithm \ref{alg:L-Slash0}} with \eqref{alg:smoothing} replaced by the meta-smoothing: $$ \mathsf u^{\ell,i} \leftarrow\mathsf u^{\ell,i-1} +\mathcal B^{\ell, i-1}_{\mathsf A_\eta, \mathsf r} (\mathsf r^{\ell,i-1}).$$

In this paper, we consider two different methods to realize the meta-smoother $\mathcal B_{\mathsf A_\eta, \mathsf r}$. The first one is based on the convolutional smoother, and its kernels are inferred from a convolutional hypernetwork. Thus, we call it the direct method, which is natural but mediocre. The second one is based on subspace correction smoother. Using subspace correction as a smoother is not as common as other smoothers such as Gauss-Seidel, but the numerical experiments show it performs better than direct methods.

\subsubsection{Direct Method} In the basic structure of PDE-MgNet, the smoothers are convolutions. Therefore, the direct method is to use Meta-NN to infer the value of these convolution kernels. We can use a vanilla DNN as the Meta-NN. \textbf{Algorithm \ref{alg:Bd}} presents the details of this method.
	\begin{algorithm}[!htb]
		\caption{$\mathcal B = B_{\text{d}}(\mathsf r,\mathsf A_\eta; \mathcal G$)}\label{alg:Bd}
		\begin{algorithmic}
			\State Hyper-parameters: $\mathcal G$
			\State Inputs: $\mathsf r$, $\mathsf A_\eta$, 
			\State Outputs: $\mathcal B$
			
			\State 1. Calculate subspace:	
			$$\mathsf B \leftarrow \mathcal G(\dfrac {\mathsf r}{||\mathsf r||}, \mathsf A_\eta). $$
			\State 2. Define the effect of $\mathcal B$ as $$\mathcal B(\mathsf r) := \mathsf B \star \mathsf r. $$
			\State \Return	$\mathcal B$.
			
		\end{algorithmic}
	\end{algorithm}
As for the structure of Meta-NN, we use a fully connect neural network with two hidden layers with 100 neurons in each layer.  

\subsubsection{Subspace Correction Method} The subspace correction (SC) method is a classical numerical method for solving linear equations. The SC smoother has more flexible parameters than the convolutional smoother in PDE-MgNet. For a linear system $\mathbf A \mathbf u = \mathbf f$, a subspace correction $\mathcal B$ is determined by a subspace $\mathbb G$ which is usually represented by the range of a matrix $\mathbf G$, i.e. $$\mathbb G = \mathrm{span}\{\mathbf g_1, \mathbf g_2,..., \mathbf g_L\}=\mathrm{range}(\mathbf G),$$ where $\mathbf G = (\mathbf g_1,...,\mathbf g_L)$.

Notice that $\mathbf g_i$ has the same dimension as $\mathbf f$ and $\mathbf u$. Thus, we let Meta-NN export multi-channel tensors with the same shape as $\mathbf f$ and then reshape each channel to form $\mathbf g_i$. Details are given by \textbf{Algorithm \ref{alg:Bsc}}. 
\begin{algorithm}[!htb]
	\caption{$\mathcal B =  B_{\text{sc}}(\mathsf r, \mathsf A_\eta; \mathcal G)$}
	\label{alg:Bsc}
	\begin{algorithmic}
		\State Hyper-parameters: Meta-NN $\mathcal G$
		\State Inputs: $\mathsf r$, $\mathsf A_\eta$
		\State Output: $\mathcal B$
		\State 1. Calculate subspace:	$$\mathsf G \leftarrow \mathcal G(\mathsf r, \mathsf A_\eta),$$
		\State where $\mathsf G$ is a tensor with shape $L\times K\times J\times I$.
		\State 2. Reshape the tensor $\mathsf G$ to $L\times KJI$ matrix, and write its transpose as $\mathbf G$, which is a $KJI\times L$ matrix.
		\State 3. Do subspace correction with the subspace $\mathbb G = \text{range}(\mathbf G)$:
		$$\mathbf S_\eta \leftarrow \mathbf A_\eta \mathbf G.$$
		$$\mathbf e = \mathbf G(\mathbf G^\top \mathbf S_\eta)^{-1}{\mathbf G^\top} \mathbf r.$$
		\State 4. Define the effect of $\mathcal B$ as
		$$
		\mathcal B(\mathsf r,\mathbf A_\eta; \mathcal G):=\text{Reshape}(\mathbf e).
		$$
		where $\text{Reshape}(\cdot)$ means to reshape $\mathbf e$ to the same shape as tensor $\mathsf r$.
		\State \Return	$\mathcal B$.
	\end{algorithmic}

\end{algorithm}
The architecture of Meta-NN $\mathcal G$ in \textbf{Algorithm \ref{alg:Bsc}} needs a more careful design than the one in \textbf{Algorithm \ref{alg:Bd}}. To select an appropriate subspace for traditional SC method is also difficult. The most popular choice is the Krylov subspace, i.e. $\mathbb G_{\mathrm K} = \{\mathbf r, \mathbf A \mathbf r, ..., \mathbf A^k\mathbf r\}$. If we write $\mathbb G_{\mathrm K} = \{\mathbf f_0(\mathbf A)\mathbf r, \mathbf f_1(\mathbf A) \mathbf r, ..., \mathbf f_k( \mathbf A)\mathbf r\}$, and $ \mathbf f_i(\mathbf A) = \mathbf A^i$. Inspired by such formulation, we design the Meta-NN in \textbf{Algorithm \ref{alg:Bsc}} as
\begin{equation}\label{Meta-NN}
\mathcal G_\theta(\mathsf r, \mathsf A_\eta)= \mathcal{N}_{\text{FC}_\theta(\mathsf A_\eta)}(\mathsf r).
\end{equation}
Here, $\mathcal N_\gamma$ is a CNN used to convert $\mathsf r$ into a multi-channel tensor, and each channel of the output plays the role as a $\mathbf f_i(\mathbf A)$ in $\mathbb G_{\mathrm K}$. The weights $\gamma$ of $\mathcal N_\gamma$ is the output of the neural network $\text{FC}_\theta$. In this paper, the CNN $\mathcal N_\gamma$ is a 3-layer Dense-Net block\citep{huang2017densely} and $\text{FC}_\theta$ is a 2-layer fully connected neural network. For the Dense-Net block, we use $l$ to represent the channel number. An illustration of this Meta-NN is given in  \textbf{Figure \ref{fig:MetaNN}}.
\begin{figure}
    \centering
    \includegraphics[width = 0.7\textwidth]{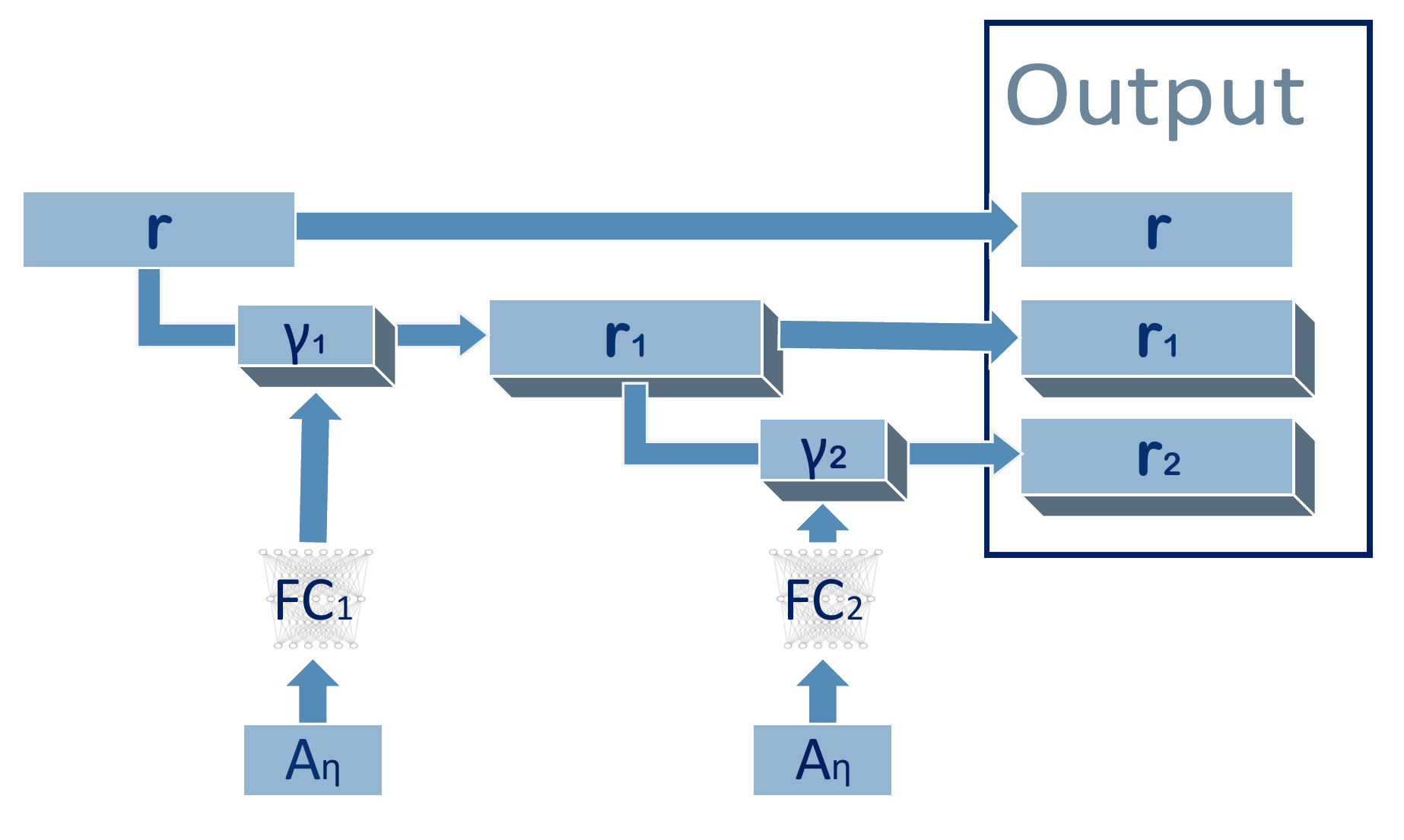}
    \caption{The architecture of Meta-NN with 3 output channels as an example.}
    \label{fig:MetaNN}
\end{figure}

\subsection{Training of Meta-MgNet.}

Suppose the distribution of $\eta$ is $Z$ and for each given $\eta$ the distribution of $\mathsf f$ is $F_\eta$. Let the number of samples of $\eta$ be $M_{\text p}$. For each $\eta$, we sample $\mathsf f$ for $M_{\text{m-train}}$ times and generate the data set $\mathfrak X_{F,\eta}$. Similar as the loss \eqref{Loss2}, we consider the following loss function
\begin{equation*}
Loss=E_{\eta\sim Z, \mathsf f \sim F_\eta} \dfrac {||\mathsf f -\mathsf A_\eta \star \mathsf u_T||^2}{||\mathsf f||^2}\approx \dfrac 1{M_{\text p}M_{\text{m-train}}}\sum_{\eta\in \mathfrak Z}\sum_{\mathsf f\in \mathfrak X_{F,\eta}}\dfrac {||\mathsf f -\mathsf A_\eta \star \mathsf u_T||^2}{||\mathsf f||^2}.
\end{equation*}
In our experiments, we choose $T=1$ following \citep{katrutsa2017deep}.
\begin{remark}
We can fine-tune the trained model when we are given a PDE with a new $\eta$, just like what meta-learning usually does. However, as shown in \textbf{Appendix I}, it brings little benefit and thus we shall omit fine-tuning. 
\end{remark}
\subsection{Convergence Analysis}

This section analyzes the convergence of the proposed Meta-MgNet \eqref{iterationmg} with SC for Poisson's equation. We assume that the discretization schemes is either FDM with 7-point (or 9-point) stencil or FEM with P1 or Q1 elements. Now suppose the approximation of $\mathbf u$ is $\mathbf u_t$, then the two-grid MG iteration can be written as
	\begin{eqnarray*}
		&\mathbf r_{t+\frac{1}{2}}&= \mathbf f-\mathbf A \mathbf u_t = \mathbf A(\mathbf u -\mathbf u_t)\\
		&\mathbf u_{t+\frac{1}{2}}&=\mathbf u_t+\mathbf P\mathbf A_c^{-1} \mathbf P^\top\mathbf r_{t+\frac{1}{2}} \\
		&\mathbf r_{t+1}&= \mathbf f-\mathbf A \mathbf u_{t+\frac{1}{2}} = \mathbf A(\mathbf u -\mathbf u_{t+\frac{1}{2}})\\\\
		&\mathbf u_{t+1} &= \mathbf u_{t+\frac{1}{2}} +\mathbf B_{\mathbf r_{t+1}} \mathbf r_{t+1}.
	\end{eqnarray*}
	Then we have the recurrence relation of the error from $t$ to $t+1$:
	\begin{equation*}
		||\mathbf u - \mathbf u_{t+1}||_{\mathbf A} = ||(\mathbf I - \mathbf B_{\mathbf r_{t+1}}\mathbf A )(\mathbf u - \mathbf u_{t+\frac 12}) ||_{\mathbf A} = ||(\mathbf I - \mathbf B_{\mathbf r_{t+1}}\mathbf A )(\mathbf I - \mathbf C)(\mathbf u - \mathbf u_{t}) ||_{\mathbf A},
	\end{equation*}
	where $\mathbf C = \mathbf P\mathbf A_c^{-1} \mathbf P^\top\mathbf A$.
	
	By \citep{hackbusch2013multi}, if the prolongation $\mathbf P$ is given by 7-point stencil or 9-point stencil, it is obviously that
	\begin{equation}\label{Prolongation1}
		||(\mathbf I -\mathbf C)\mathbf v||^2_{\mathbf A}\leq	||\mathbf v||^2_{\mathbf A}, \ \forall \mathbf v\in \mathbb V_h,
	\end{equation}
    and there is a constant $c_a>0$ s.t. $\forall \mathbf v \in \mathbb V_h$, 
	\begin{equation}
		\label{Prolongation2}
		||(\mathbf I -\mathbf C)\mathbf v||^2_{\mathbf A}\leq c_a \rho{_\mathbf A}^{-1}||\mathbf v||_{\mathbf A^2}^2,
	\end{equation} 
	where $\rho_\mathbf A$ is the spectral radius of $\mathbf A$.
	
	Now, we introduce the following assumptions (see also \citep{xu1989theory, xu1992iterative}).
	\begin{assumption}\label{assumption:convergence}
	For a given $\mathbf r$,  we assume that the associated $\mathbf B_{\mathbf r}$ satisfies:
	\begin{enumerate}
		\item $\mathbf B_{\mathbf r}$ is semi-symmetric positive defined (SSPD) and $\mathbf B_{\mathbf r} \mathbf A\mathbf B_{\mathbf r} = \mathbf B_{\mathbf r}$,
		\item There is a constant $c_s>0$ independent with $\mathbf r$, s.t. 
		\begin{equation}\label{key}
			||\mathbf r||^2\leq c_s\rho_{\mathbf A} \mathbf r^\top\mathbf B_{\mathbf r} \mathbf r.
		\end{equation}
	\end{enumerate}
	\end{assumption}
	Then, we have the following convergence theorem for Meta-MgNet (see also \citep{bank1985sharp}).
	
	\begin{theorem} Let $\{\mathbf u_{t}\}$ be the sequence generated by the Meta-MgNet \eqref{iterationmg}. Then, we have the convergence estimation
		\begin{equation*}
		||\mathbf u - \mathbf u_{t}||_{\mathbf A} \leq \delta^{\frac t2}||\mathbf u - \mathbf u_{0}||_{\mathbf A},\quad\mbox{with}\ \delta = 1-\dfrac{1}{c_a c_s}.
	\end{equation*}
	\end{theorem}
	\begin{proof}
		It suffices to show that
	\begin{equation}\label{convergence:step}
		||\mathbf u - \mathbf u_{t+1}||_{\mathbf A} \leq \delta^{\frac 12}||\mathbf u - \mathbf u_{t}||_{\mathbf A}.
	\end{equation}
	If $\mathbf B_{\mathbf{r}}$ satisfies \textbf{Assumption \ref{assumption:convergence}}, we have
	\begin{eqnarray*}
		&||\mathbf u - \mathbf u_{t+1}||^2_{\mathbf A} &= ||(\mathbf I - \mathbf B_{\mathbf r_{t+1}}\mathbf A )(\mathbf u - \mathbf u_{t+\frac 12}) ||^2_{\mathbf A} \\
		&&= (\mathbf u - \mathbf u_{t+\frac 12})^\top(\mathbf I - \mathbf A\mathbf B_{\mathbf r_{t+1}} )\mathbf A (\mathbf I - \mathbf B_{\mathbf r_{t+1}}\mathbf A ) (\mathbf u - \mathbf u_{t+\frac 12}) \quad  (\text{by } \textbf B_{\mathbf r_{t+1}} \text{is SSPD})\\
		&& =(\mathbf u - \mathbf u_{t+\frac 12})^\top(\mathbf A -2 \mathbf A\mathbf B_{\mathbf r_{t+1}} \mathbf A +\mathbf A \mathbf B_{\mathbf r_{t+1}}\mathbf A \mathbf B_{\mathbf r_{t+1}}\mathbf A ) (\mathbf u - \mathbf u_{t+\frac 12}) \\
		&&=(\mathbf u - \mathbf u_{t+\frac 12})^\top(\mathbf A - \mathbf A\mathbf B_{\mathbf r_{t+1}} \mathbf A ) (\mathbf u - \mathbf u_{t+\frac 12}) \quad (\text{by } \mathbf B_{\mathbf r_{t+1}} \mathbf A\mathbf B_{\mathbf r_{t+1}} = \mathbf B_{\mathbf r_{t+1}})\\
		&&=|| \mathbf u - \mathbf u_{t+\frac 12}||_{\mathbf A}^2 - (\mathbf u - \mathbf u_{t+\frac 12})^\top \mathbf A\mathbf B_{\mathbf r_{t+1}}\mathbf A(\mathbf u - \mathbf u_{t+\frac 12}) \\
		&&=|| \mathbf u - \mathbf u_{t+\frac 12}||_{\mathbf A}^2 - \mathbf r_{t+1}^\top\mathbf B_{\mathbf r_{t+1}} \mathbf r_{t+1}\\
		&&\leq || \mathbf u - \mathbf u_{t+\frac 12}||_{\mathbf A}^2 - c^{-1}_s\rho^{-1}_{\mathbf A} \mathbf r_{t+1}^\top \mathbf r_{t+1} \quad (\text{by } \eqref{key})\\
		&&= || \mathbf u - \mathbf u_{t+\frac 12}||_{\mathbf A}^2 - c^{-1}_s\rho^{-1}_{\mathbf A} (\mathbf u - \mathbf u_{t+\frac 12})^\top \mathbf A^2(\mathbf u - \mathbf u_{t+\frac 12})\\
		&&= || \mathbf u - \mathbf u_{t+\frac 12}||_{\mathbf A}^2 - c^{-1}_s\rho^{-1}_{\mathbf A} ||(\mathbf u - \mathbf u_{t+\frac 12})||_{\mathbf A^2}^2
	\end{eqnarray*}
	Write $\mathbf w_{t} =\mathbf u -\mathbf u_{t+\frac 12} = (\mathbf I -\mathbf C)(\mathbf u -\mathbf u_t)$, with \eqref{Prolongation1} and \eqref{Prolongation2}, we have 
	\begin{eqnarray*}
		&||\mathbf u - \mathbf u_{t+1}||^2_{\mathbf A} \leq || \mathbf w||_{\mathbf A}^2 - c^{-1}_s\rho^{-1}_{\mathbf A} ||\mathbf w||_{\mathbf A^2}\leq||\mathbf u - \mathbf u_{t}||^2_{\mathbf A} -c^{-1}_sc^{-1}_a||\mathbf u - \mathbf u_{t}||^2_{\mathbf A} = \delta ||\mathbf u - \mathbf u_{t}||^2_{\mathbf A}.
			\end{eqnarray*}
	Thus, \eqref{convergence:step} is derived.
	
	Now, we only need to verify that $\mathbf G = \mathbf G(\mathbf r)$ given in \textbf{Algorithm \ref{alg:Bsc}} stratifies \textbf{Assumption \ref{assumption:convergence}}.
	Indeed, we have $$\mathbf B_{\mathbf r} = \mathbf G (\mathbf G^\top \mathbf A \mathbf G)^{-1} \mathbf G^\top.$$
	For the first assumption, 
	$$ \mathbf A\text{ is SPD}\Rightarrow \mathbf G^\top \mathbf A \mathbf G \text{ is SPD} \Rightarrow (\mathbf G^\top \mathbf A \mathbf G)^{-1} \text{ is SPD} \Rightarrow \mathbf B_{\mathbf r} =\mathbf G (\mathbf G^\top \mathbf A \mathbf G)^{-1} \mathbf G^\top \text{ is SSPD}.$$
	And  $$\mathbf B_{\mathbf r}  \mathbf A \mathbf B_{\mathbf r}=\mathbf G (\mathbf G^\top \mathbf A \mathbf G)^{-1} \mathbf G^\top\mathbf A\mathbf G (\mathbf G^\top \mathbf A \mathbf G)^{-1} \mathbf G^\top = \mathbf G (\mathbf G^\top \mathbf A \mathbf G)^{-1} \mathbf G^\top =\mathbf B_{\mathbf r}.$$
	
	For the second assumption, without loss of generality, we assume $||\mathbf r||=1$, then we have $||\mathbf r||^2_\mathbf{A}\leq \rho_{\mathbf{A}}||\mathbf r||^2 = \rho_{\mathbf{A}}$.
    
    Write $\mathbf G = [\mathbf g_1,...,\mathbf g_L]$, and $\mathbf g_i, i = 1,2,...,L$ satisfy $\mathbf g_i^\top\mathbf A \mathbf g_j = \delta_{ij}$.
    Choose $\mathbf g_1 = \dfrac{\mathbf r}{||\mathbf r||_\mathbf{A}}$. Then, we obtain $$\mathbf r^\top\mathbf{B}_{\mathbf r}\mathbf r\geq (\mathbf r^\top \mathbf g_1)^2 =\dfrac{1}{||\mathbf r||^2_\mathbf{A}}\geq \frac{1}{\rho_{\mathbf A}}.$$ This means that \eqref{key} is satisfied if $\mathbf r\in \text{range}(\mathbf G)$, which is obvious from the design of the Meta-NN for $B_{\text{sc}}$.
	\end{proof}
	
\section{Numerical Experiments}
In this section, we evaluate the performance of Meta-MgNet through a series of numerical experiments. In \textbf{Section 5.1} and \textbf{5.2}, we apply Meta-MgNet to 2D and 3D anisotropic diffusion equations on domain $\Omega=[0,1]^d, d=2,3$. In \textbf{Section 5.3}, we demonstrate why it is challenging to train a PDE-MgNet that generalizes well for all $\eta$. In \textbf{Section 5.4}, we include more practical examples. 

In the experiments, we compare Meta-MgNet with PDE-MgNet and MG method. Both Meta-MgNet and PDE-MgNet are trained on a data set created by a set of $\eta$, which will be later described in detail. Furthermore, we also train PDE-MgNet for each individual $\eta$ and denote the trained model as PDE-MgNet-$\eta$. During testing, we only apply PDE-MgNet-$\eta$ on the test data generated by the same $\eta$ as during training. Thus, PDE-MgNet-$\eta$ demonstrates the best accuracy of PDE-MgNet may achieve while it is impractical since it requires retraining of PDE-MgNet for each individual $\eta$. For the classical MG method, we choose the Krylov subspace smoother, GS smoother, line-GS smoother, and damped Jacobi smoother. Since the GS and line-GS smoother are challenging to implement efficiently with GPU, we use Matlab on CPU instead. All other algorithms are implemented in PyTorch on GPU.

\subsection{2D Anisotropic Diffusion Equations}
We consider the anisotropic diffusion equation
\begin{equation}
	\left\{
	\begin{split}
		-\nabla\cdot(C \nabla u) &= f, &\text{ in } \Omega,\\
		u &= 0,  &\text{ on } \partial \Omega,
	\end{split}
	\right.
\end{equation}
where $C = C(\epsilon, \theta)= \begin{pmatrix}
	\cos \theta& -\sin\theta\\\sin\theta&\cos \theta
\end{pmatrix}
\begin{pmatrix}
	1& 0\\0&\epsilon
\end{pmatrix}
\begin{pmatrix}
	\cos \theta& \sin\theta\\-\sin\theta&\cos \theta
\end{pmatrix}$ is a $2\times 2$ matrix, $\epsilon<1, \theta\in[0,\pi].$
\subsubsection{Settings: training, testing and hyperparameter selection}
\begin{enumerate}[(1)]
	\item For the training set, we randomly sample $M_{\text p}= 20$ different sets of parameters $\eta$ from an interval $I_\text {train}$. The interval $I_\text {train}$ is different for different PDEs. Thus, we will specify it later for each experiment. For each given $\eta$, we randomly sample $M_{\text{m-train}} = 100$ right-hand-side function, and each entry of $\mathsf f$ is sampled from the Gaussian distribution $N(0,1)$. We use ADAM method and the unsupervised learning loss \eqref{Loss2} to train PDE-MgNet for $50$ epochs and Meta-MgNet for $20$ epochs. Since no matter if Meta-MgNet is fine-tuned, the result is almost the same, we skip the fine-tuning stage of Meta-MgNet. (See \textbf{Appendix I} for an ablation study on fine-tuning.) The learning rate for ADAM is $0.02$ and the batch size is $64$.
	\item For the test set, we choose some specific $\eta$ from a set $I_{\text {test}}$, and for each $\eta$ we randomly sample  $M_{\text{m-test}} = 10$ right-hand-side function. The stopping criterion for all compared algorithms is chosen as $$\dfrac {||\mathbf f-\mathbf A_{\eta}\mathbf u_t||_2}{||\mathbf f||_2}<10^{-6}.$$ Number of iterations and wall time are used as metrics to compare the performance of different algorithms. Furthermore, we use "mean$\pm$std" to show the average and standard deviation of the number of iteration and wall time over the $M_{\text{m-test}}$ samples. 
	\item We use $N\times N$ rectangular mesh and  $Q_1$-element to discretize the PDEs. We select $N=256$ and the number of layers $J=5$. The prolongations $\mathcal P$ and restrictions $\mathcal R$ are given by the traditional 9-point stencil, i.e. the kernel of $\mathcal P$ and $\mathcal R$ is
		$$\mathsf P = \mathsf R =\begin{pmatrix}
	\frac 1 4 & \frac 1 2& \frac 14\\
	\frac 1 2 & 1 & \frac 12\\
	\frac 1 4 & \frac 1 2& \frac 14\\
	\end{pmatrix}.$$
	The $\mathcal A^\ell$ on coarse grid is also given by the geometric multigrid method, and it is easy to verify that each $\mathsf A^\ell$ is equal for $Q_1$-element. We chose $\backslash$-Cycle structure for all algorithms with $\nu_1 = 2, \nu_2 = \cdots = \nu_J = 1$. (We tried several different settings of $\nu_l$ and found that $\nu_1 = 2, \nu_2 = \cdots = \nu_J = 1$ is the best option. See \textbf{Appendix II} for the corresponding ablation study.)

	\item For PDE-MgNet, the kernel size of smoother at each layer is $7\times7$. For Meta-MgNet, we choose SC method $B_{\text {sc}}$ in \textbf{Algorithm \ref{alg:Bsc}} and use the Meta-NN in \eqref{Meta-NN}. The kernel size of the output is $7\times 7$ and the increase channel number of Meta-NN is $l=3$. Note that a comparison between $B_{\text {d}}$ and $B_{\text {sc}}$ is given in \textbf{Appendix III}, it shows that $B_{\text {sc}}$ is the better option. Thus, we only present the numerical result of $B_{\text {sc}}$ in the main body of this paper. 

\end{enumerate}

We compare number of iterations and wall time of Meta-MgNet, PDE-MgNet and MG method. We consider two different generalization scenarios that will be called ``in-distribution generalization" and ``out-of-distribution (OoD) transfer". For in-distribution generalization, we have $I_{\text {test}} \subset I_{\text {train}}$, while for OoD transfer, the parameter $\eta$ of the test data is entirely outside the interval of the training data, i.e. $I_{\text {test}} \subset I_\text {train}^c$. 

\subsubsection{In-distribution generalization}

In this group of experiments, the training set of PDE-MgNet and Meta-MgNet is generated by fixing $\theta = 0$ in \textbf{Table \ref{table:AnisotropicPoisson}} and $\theta = 0.1\pi$ in \textbf{Table \ref{table:AnisotropicPoisson1.1}} and randomly sampling $\epsilon$ with distribution $\lg \frac 1\epsilon \sim U[0,5]$. The \textbf{Table \ref{table:AnisotropicPoisson}}, \textbf{Table \ref{table:AnisotropicPoisson1.1}} and \textbf{Figure \ref{fig:AnisotropicPoisson}} show Meta-MgNet has overall better performance than PDE-MgNet and MG methods, while the advantage is significant when $\epsilon$ is small. It is worth mentioning that the line-GS smoother can only be applied to several specific $\theta$, such as $0, \dfrac{\pi}{4}, \dfrac{\pi}{2}$, and $\epsilon$ should be small enough, thus line-GS smoother is limited in practical applications. 

\begin{table}
\footnotesize
	\begin{tabular}{|c|ccccccc|}\hline
		\#iterations&Meta-MgNet  	& PDE-MgNet 	 & PDE-MgNet-$\eta$ &MG(Krylov)        &MG(GS)&MG(line-GS)          &MG(Jacobi)  \\\hline
		$\epsilon=1$		&$  4.0\pm 0.00$&-          	 &$   7.0\pm 0.00$  &$4.0\pm0.00$     &$  10.0\pm 0.00$&-&$15.0\pm0.00$  \\
		$\epsilon=10^{-1}$	&$  7.5\pm 0.50$&  $19.2\pm0.40$ &$  21.2\pm 0.60$  & $7.9\pm0.30$    &$  33.7\pm 0.48$&-&$90.2\pm0.98$ \\
		$\epsilon=10^{-2}$	&$ 35.1\pm 1.04$& $178.9\pm2.74$ &$ 149.7\pm 3.44$  &$52.5\pm0.81$    &$ 253.6\pm 4.19$&$553.6\pm27.56$&$752.8\pm12.23$\\
		$\epsilon = 10^{-3}$&$171.6\pm 6.34$&$1.2\text{e}3\pm12.85$&$ 910.9\pm15.64$  &$345.9\pm3.88$   &$1.9\text{e}3\pm25.56$&$62.3\pm1.76$&$5.6\text{e}3\pm119.42$  \\
		$\epsilon = 10^{-4}$&$375.2\pm 5.88$&              - &$3.1\text{e}3\pm35.70$  &$2.2\text{e}3\pm27.94$ &-              &$11.0\pm0.00$ &-\\
		$\epsilon = 10^{-5}$&$797.8\pm12.76$&-               &$9.9\text{e}3\pm40.81$  &$7.6\text{e}3\pm81.96$ &-               &$11.0\pm0.00$&-\\
		\hline
		wall time 			&  			    &                &                  &                  &                & 	\\\hline
		$\epsilon = 1$		&$0.03\pm0.00$  &-               &$\mathbf{0.02\pm0.00}$     &$\mathbf{0.02\pm0.00}$     &$ 0.14\pm0.01$&-  &$0.04\pm0.00$\\
		$\epsilon = 10^{-1}$&$0.05\pm0.00$  &$0.05\pm0.00$   &$0.06\pm0.00$     &$\mathbf{0.04\pm0.00}$     &$ 0.48\pm0.02$ &- &$0.23\pm0.00$ \\
		$\epsilon = 10^{-2}$&$\mathbf{0.22\pm0.01}$  &$0.44\pm0.01$   &$0.37\pm0.01$     &$0.25\pm0.00$     &$ 3.47\pm0.15$ &$12.6\pm0.86$ &$1.85\pm0.03$   \\
		$\epsilon = 10^{-3}$&$\mathbf{1.06\pm0.04}$  &$3.04\pm0.03$   &$2.28\pm0.05$     &$1.64\pm0.02$     &$27.33\pm0.68$ &$1.40\pm0.04$ &$13.95\pm0.29$\\
		$\epsilon = 10^{-4}$&$2.31\pm0.03$  &-               &$7.69\pm0.13$     &$10.56\pm0.14$    &-             &$\mathbf{0.27\pm0.01} $  &-\\
		$\epsilon = 10^{-5}$&$4.91\pm0.08$  &-               &$24.49\pm0.14$    &$35.67\pm0.40$    &-              &$\mathbf{0.27\pm0.02} $ &-\\
		\hline
	\end{tabular}
	\caption{The number of iterations (when the stopping criteria is met) and the wall time of Meta-MgNet, PDE-MgNet and MG method.  "-" means the algorithm does not converge within $10^4$ iterations.}\label{table:AnisotropicPoisson}
\end{table}
\begin{figure}
    \centering
    \includegraphics{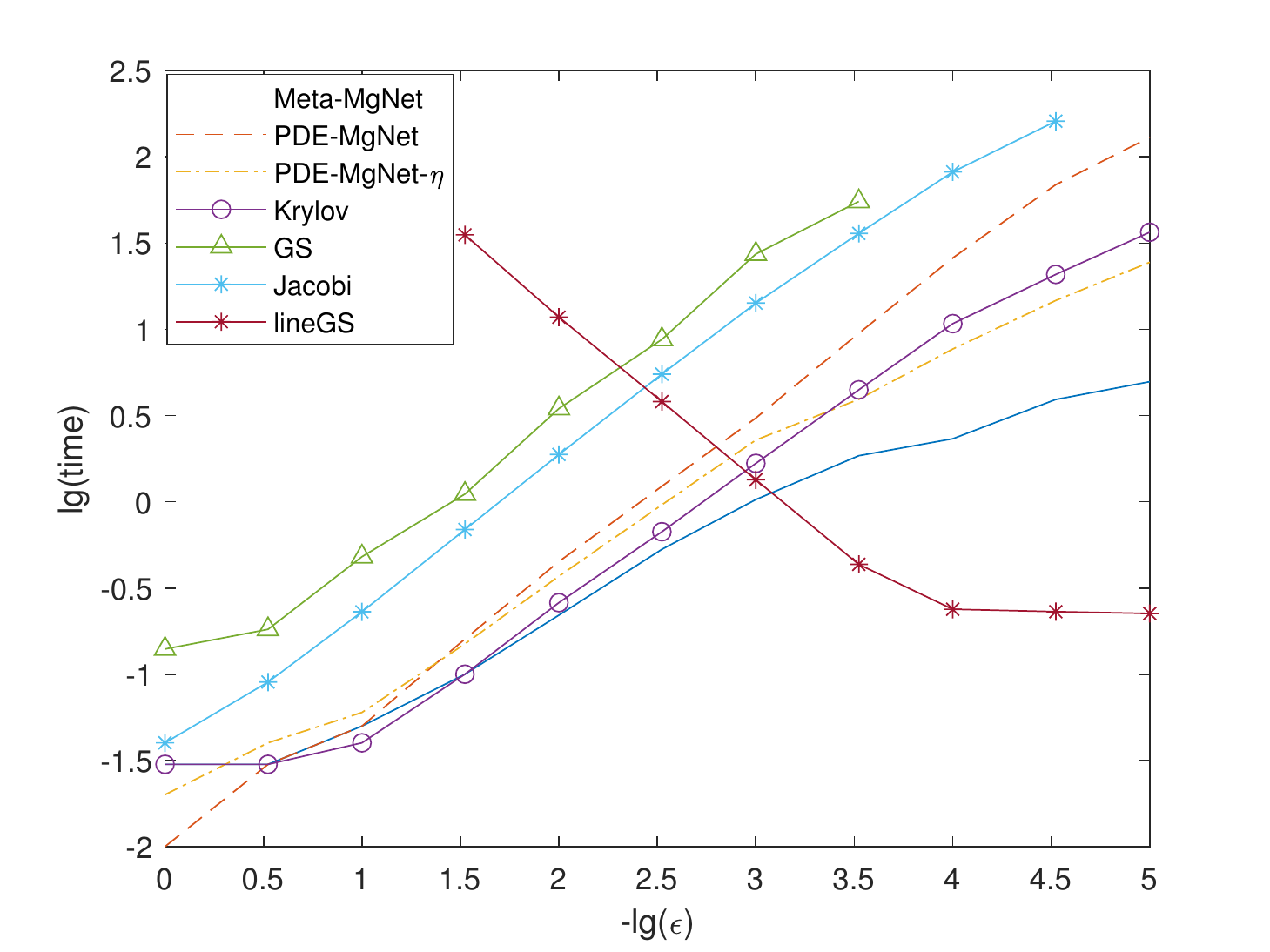}
    \caption{Wall time of Meta-MgNet, PDE-MgNet, and the MG methods while $\theta=0$ and $\epsilon$ varies.}
    \label{fig:AnisotropicPoisson}
\end{figure}

\begin{table}
\footnotesize
	\begin{tabular}{|c|ccccccc|}\hline
		\#iterations        &Meta-MgNet  	& PDE-MgNet 	   & PDE-MgNet-$\eta$       &MG(Krylov)             &MG(GS)                 &MG(line-GS)          &MG(Jacobi)  \\\hline
		$\epsilon=1$		&$4.0\pm0.00 $  &-         	   &$17.00\pm0.00$       &$4.0\pm0.00$           &$  10.0\pm 0.00$       &- &$15.0\pm0.00$  \\
		$\epsilon=10^{-1}$	&$5.4\pm0.49 $  &$136.7\pm1.95$   &$68.30\pm2.69$        &$8.0\pm0.30$           &$ 27.3\pm0.48$         &- &$64.70\pm0.64$  \\
		$\epsilon=10^{-2}$	&$28.5\pm0.50 $ &$1.0\text{e}3\pm25.39$    &$861.40\pm24.53$         &$45.4\pm0.49$          &$187.0\pm2.53$         &- &$476.30\pm4.78$ \\
		$\epsilon=10^{-3}$  &$94.2\pm1.33 $ &$3.8\text{e}3\pm183.22$   &$1.9\text{e}3\pm60.41$       &$142.8\pm1.89$         &$707.6\pm15.24$        &- &$1.8\text{e}3\pm44.47$  \\
		$\epsilon=10^{-4}$  &$129.4\pm2.42$ &$5.3\text{e}3\pm180.33$   &$3.8\text{e}3\pm94.79$   &$187.8\pm3.79$         &$990.2\pm19.03 $       &- &$2.5\text{e}3\pm62.01$\\
		$\epsilon=10^{-5}$  &$134.8\pm2.86$ &$5.6\text{e}3\pm168.51$   &$4.1\text{e}3\pm110.96$   &$195.9\pm2.43$         &$1.0\text{e}3\pm37.45$ &- &$2.6\text{e}3\pm79.24$     \\
		\hline
		wall time 			&  			    &                  &                        &                       &                       &                   &	\\\hline
		$\epsilon=1$		&$0.04\pm0.00 $ &-                 &$0.05\pm0.00$  &$\mathbf{0.03\pm0.00}$ &$ 0.14\pm0.01$         &- &$0.04\pm0.00$\\
		$\epsilon=10^{-1}$  &$\mathbf{0.05\pm0.00} $ &$0.36\pm0.01$     &$0.18\pm0.01$          &$\mathbf{0.05\pm0.00}$ &$0.29\pm0.02$          &- &$0.17\pm0.00$ \\
		$\epsilon=10^{-2}$  &$\mathbf{0.21\pm0.00} $ &$2.64\pm0.06$     &$2.23\pm0.06$            &$0.24\pm0.00$          &$ 2.05\pm0.08$         &- &$1.26\pm0.01$   \\
		$\epsilon=10^{-3}$  &$\mathbf{0.69\pm0.01} $ &$9.89\pm0.53$     &$4.90\pm0.17$            &$0.75\pm0.01$          & $7.78\pm0.21 $        &- &$4.79\pm0.11$\\
		$\epsilon=10^{-4}$  &$\mathbf{0.94\pm0.02} $ &$13.87\pm0.48$    &$9.90\pm0.24$            &$0.98\pm0.02$          &$10.76\pm0.25$         &- &$6.66\pm0.15$  \\
		$\epsilon=10^{-5}$  &$\mathbf{0.98\pm0.02}$  &$14.56\pm0.44$    &$10.69\pm0.30$         &$1.02\pm0.01$          &$11.36\pm0.51$         &- &$7.03\pm0.21$\\
		\hline
	\end{tabular}
	\caption{The number of iterations (when the stopping criteria is met) and the wall time of Meta-MgNet, PDE-MgNet and MG method.  "-" means the algorithm does not converge within $10^4$ iterations.}\label{table:AnisotropicPoisson1.1}
\end{table}

\subsubsection{Out-of-distribution (OoD) transfer}

The first group of experiments are with the fixed $\theta=0$. The training set is generated with $\lg \frac 1\epsilon \sim U[2,3]$ and the test set is generated with $\epsilon = 1, 10^{-1}, 10^{-4}, 10^{-5}$ (i.e. a test for OoD transfer with respect to $\epsilon$). For the second group of experiments, the training set is generated with $\theta \sim U[0.125\pi, 0.375\pi]$ and $\lg \frac 1\epsilon \sim U[0,5]$, while the test set is generated with $\theta = 0.05\pi, 0.12\pi, 0.4\pi$ and $0.5 \pi$ (i.e. a test for OoD transfer with respect to $\theta$).

\textbf{Table \ref{table:AnisotropicPoisson1}} shows that Meta-MgNet has superior ability of OoD transfer with respect to $\epsilon$ in comparison with PDE-MgNet, while \textbf{Table \ref{table:AnisotropicPoisson2}} shows that Meta-MgNet is noticeably superior in OoD transfer with respect to $\theta$ in comparison with PDE-MgNet. Note that classical methods are not learning-based, and hence they do not have the issue of OoD transfer. Thus, the results of MG methods in Table \ref{table:AnisotropicPoisson1} are copied from Table \ref{table:AnisotropicPoisson}.

\begin{table}
\footnotesize
	\begin{tabular}{|c|ccccccc|}\hline
		\#iterations&Meta-MgNet  	& PDE-MgNet 	 & PDE-MgNet-$\eta$ &MG(Krylov)        &MG(GS)      &MG(line-GS)    &MG(Jacobi)  \\\hline
		$\epsilon=1$		&$7.0\pm0.00$    &-          	 &$   7.0\pm 0.00$  &$4.0\pm0.00$     &$  10.0\pm 0.00$&-&$15.0\pm0.00$\\
		$\epsilon=10^{-1}$	&$10.0\pm0.00$   &$23.0\pm0.00$  &$  21.2\pm 0.60$  & $7.9\pm0.30$    &$  33.7\pm 0.48$&-&$90.2\pm0.98$\\
		$\epsilon = 10^{-4}$&$340.7\pm3.52$  &$5.8\text{e}3\pm121.90$  &$3.1\text{e}3\pm35.70$  &$2.2\text{e}3\pm27.94$ &-&$11.0\pm0.00$               &-\\
		$\epsilon = 10^{-5}$&$817.2\pm97.97$ &-               &$9.9\text{e}3\pm40.81$  &$7.6\text{e}3\pm81.96$ &-&$11.0\pm0.00$               &-\\
		\hline
		wall time 			&  			    &                &                  &                  &                && 	\\\hline
		$\epsilon = 1$		&$0.05\pm0.00$  &-               &$\mathbf{0.02\pm0.00}$     &$\mathbf{0.02\pm0.00}$     &$ 0.14\pm0.01$ &- &$0.04\pm0.00$\\
		$\epsilon = 10^{-1}$&$0.07\pm0.00$  &$0.06\pm0.00$   &$0.06\pm0.00$     &$\mathbf{0.04\pm0.00}$     &$ 0.48\pm0.02$&-  &$0.23\pm0.00$\\
		$\epsilon = 10^{-4}$&$2.08\pm0.02$  &$14.38\pm0.32$  &$7.69\pm0.13$     &$10.56\pm0.14$    &-   &$\mathbf{0.27\pm0.01}$          &-\\
		$\epsilon = 10^{-5}$&$4.99\pm0.59$  &-               &$24.49\pm0.14$    &$35.67\pm0.40$    &-   &$\mathbf{0.27\pm0.02}$          &-\\
		\hline
	\end{tabular}
	\caption{The mean and std of the number of iterations (when the stopping criteria is met) and the wall time of Meta-MgNet, PDE-MgNet and MG method on the testing set.  "-" means the algorithm does not converge within $10^4$ iterations.}\label{table:AnisotropicPoisson1}
\end{table}
\begin{table}
	\begin{tabular}{|c|cccc|}\hline
		\#iterations&Meta-MgNet,$\theta=0.05\pi$&PDE-MgNet,$\theta=0.05\pi$ &Meta-MgNet,$\theta=0.12\pi$&PDE-MgNet,$\theta =0.12 \pi$\\\hline
		$\epsilon=1$		&$3.0\pm0.00$   &-              &$3.0\pm0.00$  &-\\
		$\epsilon=10^{-1}$	&$10.6\pm0.49$  &-			     &$10.1\pm0.30$ &$132.6\pm3.10$\\
		$\epsilon=10^{-2}$	&$71.5\pm1.57$  &-			     &$72.0\pm1.61$ &$566.3\pm11.36$\\
		$\epsilon = 10^{-3}$&$322.4\pm7.03$ &-				 &$233.2\pm7.49$&$2.1\text{e}3\pm55.09$\\
		$\epsilon = 10^{-4}$&$526.7\pm14.64$&-				 &$306.0\pm7.78$&$2.8\text{e}3\pm150.40$\\
		$\epsilon = 10^{-5}$&$557.4\pm14.72$&-				 &$314.0\pm4.86$&$2.8\text{e}3\pm33.31$\\
		\hline
		wall time 				& 			  &  			& 			   & 	\\\hline
		$\epsilon = 1$		&$0.03\pm0.00$&-            &$0.03\pm0.00$ &-\\
		$\epsilon = 10^{-1}$&$0.07\pm0.00$&-&$0.07\pm0.00$&$0.35\pm0.01$\\
		$\epsilon = 10^{-2}$&$0.46\pm0.01$&-&$0.46\pm0.01$ &$1.49\pm0.04$\\
		$\epsilon = 10^{-3}$&$2.05\pm0.04$&-& $1.48\pm0.05$&$5.42\pm0.18$\\
		$\epsilon = 10^{-4}$&$3.34\pm0.09$&-&$1.95\pm0.05$&$7.49\pm0.45$\\
		$\epsilon = 10^{-5}$&$3.54\pm0.09$&-&$1.99\pm0.03$&$7.49\pm0.15$\\
		\hline
		\#iterations&Meta-MgNet,$\theta=0.4\pi$&PDE-MgNet,$\theta=0.4\pi$ &Meta-MgNet,$\theta=0.5\pi$&PDE-MgNet,$\theta =0.5 \pi$\\\hline
		$\epsilon=1$		&$3.0\pm0.00$ &-              &$3.0\pm0.0$ &-\\
		$\epsilon=10^{-1}$	&$9.0\pm0.00$&$51.7\pm1.27$&$8.9\pm0.30$&$49.3\pm1.42$\\
		$\epsilon=10^{-2}$	&$65.2\pm1.54$&$434.5\pm7.75$&$53.5\pm1.12$&$428.8\pm12.83$\\
		$\epsilon = 10^{-3}$&$240.3\pm3.41$&$1.7\text{e}3\pm50.92$&$262.9\pm5.96$&$2.8\text{e}3\pm16.81$\\
		$\epsilon = 10^{-4}$&$327.5\pm5.33$&$2.4\text{e}3\pm67.84$&$526.4\pm25.51$&-\\
		$\epsilon = 10^{-5}$&$339.5\pm6.92$&$2.5\text{e}3\pm85.99$&$908.7\pm27.43$&-\\
		\hline
		wall time 				& 			  &  			& 			   & 	\\\hline
		$\epsilon = 1$		&$0.03\pm0.00$&-            &$0.03\pm0.00$ &-\\
		$\epsilon = 10^{-1}$&$0.06\pm0.00$&$0.14\pm0.00$& $0.06\pm0.00$& $0.13\pm0.01$\\
		$\epsilon = 10^{-2}$&$0.42\pm0.01$&$1.15\pm0.03$& $0.35\pm0.01$&$1.13\pm0.03$\\
		$\epsilon = 10^{-3}$&$1.53\pm0.02$&$4.48\pm0.12$&$1.67\pm0.04$&$7.34\pm0.11$\\
		$\epsilon = 10^{-4}$&$2.09\pm0.03$&$6.38\pm0.22$& $3.35\pm0.16$&-\\
		$\epsilon = 10^{-5}$&$2.15\pm0.04$&$6.61\pm0.31$&$5.76\pm0.18$&-\\
		\hline
	\end{tabular}
	\caption{The mean and std of the number of iterations (when the stopping criteria is met) and the wall time of Meta-MgNet, PDE-MgNet and MG method on the testing set.  "-" means the algorithm does not converge within $10^4$ iterations.}\label{table:AnisotropicPoisson2}
\end{table}

\subsubsection{Further discussions}

It is also worth noting that, for both the scenarios of in-distribution generalization and OoD transfer, Meta-MgNet even outperforms PDE-MgNet-$\eta$ which uses the exact same $\eta$ for both training and testing. This shows the benefit of treating the problem of solving parameterized PDEs as a multi-task learning problem. From what it seems, the hypernetwork, i.e. Meta-NN, is able to extract certain common structure hidden within the tasks which helps with solving each individual task.

The training time of Meta-MgNet is around 45 minutes, while it is around 10 minutes for PDE-MgNet-$\eta$ for each $\eta$. Although training Meta-MgNet takes more times than PDE-MgNet-$\eta$ for each $\eta$. In practical application, especially multi-query scenarios, the utility of PDE-MgNet-$\eta$ can be significantly reduced due to the constant retraining. Therefore, Meta-MgNet has a clear overall advantage. 

\subsection{3D Anisotropic Diffusion Equations}

Consider the following 3D anisotropic diffusion equation
\begin{equation*}
\left\{
\begin{split}
-\nabla\cdot(C \nabla u) &= f, &\text{ in } \Omega,\\
u &= 0,  &\text{ on } \partial \Omega,
\end{split}
\right.\qquad\mbox{and}\qquad 
C = 
\begin{pmatrix}
\epsilon_0& &\\&\epsilon_1&\\&&\epsilon_2
\end{pmatrix},
\end{equation*}
with $\Omega = [0,1]^3$ and $\epsilon_0, \epsilon_1, \epsilon_2>0$. Without loss of generality, we set $\epsilon_0 = 1$.

\subsubsection{Settings: training, testing and hyperparameter selection}
\begin{enumerate}[(1)]
	\item The settings of training and testing are the same as the 2D case.
	\item We use $N\times N\times N$ rectangular mesh and use $Q_1$-element to discretize the PDE. Let $N=64$ and the number of layers $J=4$. The prolongations $\mathcal P$ and restrictions $\mathcal R$ are given by the traditional 9-point stencil. The $\mathcal A^\ell$ on the coarse grid is given by the geometry MG method. It is easy to verify that each $\mathsf A^\ell$ is equal for $Q_1$-element. We chose $\backslash$-Cycle structure of all algorithms with $\nu_1 =2 ,\nu_2 = \cdots = \nu_J = 1$. 
	
	\item For MgNet, the size of the kernel for the convolution smoother is $3\times3\times3$. For Meta-MgNet, we choose the SC method $B_{\text {sc}}$ in \textbf{Algorithm \ref{alg:Bsc}} and use the Meta-NN in \eqref{Meta-NN}. We simply set the convolution smoother to be one layer CNN without activation, and the size of the kernel for the output is $7\times 7\times 7$. We set the number of channels of Meta-NN to $l=3$.
\end{enumerate}

\subsubsection{In-distribution generalization}

In this group of experiments, the training data set is generated by sampling $\epsilon_1$ and $\epsilon_2$ from distribution $\lg {\dfrac 1{\epsilon_1}} \sim U[0,5]$ and $\lg {\dfrac 1{\epsilon_2}} \sim U[0,5]$ respectively.  \textbf{Table \ref{table:AnisotropicPoisson3d}} shows that Meta-MgNet is more efficient than classic MG methods.
\begin{table}
	\begin{tabular}{|c|cccccc|}\hline
		\#iterations	&Meta-MgNet		   &PDE-MgNet	&PDE-MgNet-$\eta$ &MG(Krylov)&MG(GS)	     &MG(Jacobi)  \\\hline
		$(\epsilon_1,\epsilon_2)=(10^{-1},10^{-1})$&$5.0\pm0.00$&$11.0\pm0.00$          & $11.0\pm0.00$& $7.0\pm0.00$     & $53.0\pm1.05$   &-    \\
		$(\epsilon_1,\epsilon_2)=(10^{-1},10^{-2})$&$13.0\pm0.00$    &$91.4\pm1.20$            &$46.6\pm1.69$  &$38.3\pm1.00$  & $159.9\pm5.27$ &-  \\
		$(\epsilon_1,\epsilon_2)=(10^{-1},10^{-5})$&$156.2\pm3.57$   &$3.2\text{e}3\pm73.59$         &$475.0\pm8.00$ &$606.8\pm30.08$    &  -   &-   \\
		$(\epsilon_1,\epsilon_2)=(10^{-2},10^{-2})$&$10.3\pm0.46$&$116.4\pm0.49$ &$54.3\pm0.78$&$45.7\pm0.78$&  $178.4\pm2.54 $      &$291.52\pm9.14$    \\
		$(\epsilon_1,\epsilon_2)=(10^{-2},10^{-5})$&$73.4\pm0.80$   & $3.5\text{e}3\pm77.82$       &$771.5\pm9.39$&$631.70\pm16.64$ &   -       &-\\
		$(\epsilon_1,\epsilon_2)=(10^{-5},10^{-5})$&$111.6\pm0.92$&$8.3\text{e}3\pm65.18$      &$5.5\text{e}3\pm99.85$ &$1.8\text{e}3\pm20.23$&   -       &-\\
		\hline
		wall time 			 &&&  &&& 	\\\hline
		$(\epsilon_1,\epsilon_2)=(10^{-1},10^{-1})$&$0.13\pm0.00$&$0.09\pm0.00$         &$0.09\pm0.00$ &$\mathbf{0.07\pm0.00}$ &   $4.29\pm0.10$      -&-  \\
		$(\epsilon_1,\epsilon_2)=(10^{-1},10^{-2})$&$\mathbf{0.30\pm0.00}$  &$0.65\pm0.01$        &$0.34\pm0.02$& $0.32\pm0.01$& $14.2\pm0.64$      &- \\
		$(\epsilon_1,\epsilon_2)=(10^{-1},10^{-5})$&$3.40\pm0.08$ &$22.53\pm0.51$       &$\mathbf{3.34\pm0.06}$&$4.93\pm0.25$  & -       &-   \\
		$(\epsilon_1,\epsilon_2)=(10^{-2},10^{-2})$&$\mathbf{0.25\pm0.01}$ &$0.82\pm0.00$     &$0.39\pm0.01$&$0.38\pm0.01$    &    $14.67\pm0.61$        &$6.00\pm0.00$  \\
		$(\epsilon_1,\epsilon_2)=(10^{-2},10^{-5})$&$\mathbf{1.62\pm0.02}$&$24.73\pm0.55$      &$5.42\pm0.06$ &$5.12\pm0.13$   &    -        &-\\
		$(\epsilon_1,\epsilon_2)=(10^{-5},10^{-5})$&$\mathbf{2.45\pm0.02}$ & $58.87\pm0.47$      &$38.93\pm0.71$&$14.90\pm0.16$   &   -       &-\\
		\hline
	\end{tabular}
	\caption{The mean and std of the number of iterations (when the stopping criteria is met) and the wall time of Meta-MgNet, PDE-MgNet and MG method on the testing set.  "-" means the algorithm does not converge within $10^4$ iterations.}\label{table:AnisotropicPoisson3d}
\end{table}

\subsubsection{Out-of-distribution (OoD) transfer}

In this group of experiments, the training data set is generated by sampling $\epsilon_1$ and $\epsilon_2$ from distribution $\lg {\dfrac 1{\epsilon_1}} \sim U[3,4]$ and $\lg {\dfrac 1{\epsilon_2}} \sim U[3,4]$ respectively. \textbf{Table \ref{table:AnisotropicPoisson3de}} shows that Meta-MgNet has an overall best performance, and PDE-MgNet has trouble in OoD transfer with respect to $\epsilon_1$ and $\epsilon_2$.

\begin{table}
	\begin{tabular}{|c|cccccc|}\hline
		\#iterations	&Meta-MgNet		   &PDE-MgNet	&PDE-MgNet-$\eta$ &MG(Krylov)&MG(GS)	     &MG(Jacobi)  \\\hline
		$(\epsilon_1,\epsilon_2)=(10^{-1},10^{-1})$&$10.0\pm0.00$   & -         & $11.0\pm0.00$& $7.0\pm0.00$     & $53.0\pm1.05$   &-    \\
		$(\epsilon_1,\epsilon_2)=(10^{-1},10^{-2})$&$43.1\pm0.54$     &-            &$46.6\pm1.69$  &$38.3\pm1.00$  & $159.9\pm5.27$ &-  \\
		$(\epsilon_1,\epsilon_2)=(10^{-1},10^{-5})$&$755.4\pm55.88$   & -        &$475.0\pm8.00$ &$606.8\pm30.08$    & -    &-   \\
		$(\epsilon_1,\epsilon_2)=(10^{-2},10^{-2})$&$11.0\pm0.00$  &- &$54.3\pm0.78$&$45.7\pm0.78$&   $178.4\pm2.54 $       &$291.52\pm9.14$    \\
		$(\epsilon_1,\epsilon_2)=(10^{-2},10^{-5})$&$106.7\pm1.35$      &-       &$771.5\pm9.39$&$631.7\pm16.64$ &   -       &-\\
		$(\epsilon_1,\epsilon_2)=(10^{-5},10^{-5})$&$125.7\pm2.19$ &-     &$5.5\text{e}3\pm99.85$ &$1.8\text{e}3\pm20.23$&   -       &-\\
		\hline
		wall time 			 &&&  &&& 	\\\hline
		$(\epsilon_1,\epsilon_2)=(10^{-1},10^{-1})$&$0.24\pm0.01$& -       &$0.09\pm0.00$ &$\mathbf{0.07\pm0.00}$ & $4.29\pm0.10$        -&-  \\
		$(\epsilon_1,\epsilon_2)=(10^{-1},10^{-2})$&$0.95\pm0.01$   &-        &$0.34\pm0.02$&$\mathbf{ 0.32\pm0.01}$& $14.2\pm0.64$       &- \\
		$(\epsilon_1,\epsilon_2)=(10^{-1},10^{-5})$&$16.37\pm1.22$   & -       &$\mathbf{3.34\pm0.06}$&$4.93\pm0.25$  &  -      &-   \\
		$(\epsilon_1,\epsilon_2)=(10^{-2},10^{-2})$&$\mathbf{0.26\pm0.00}$ &-     &$0.39\pm0.01$&$0.38\pm0.01$    &  $14.67\pm0.61$          &$6.00\pm0.00$  \\
		$(\epsilon_1,\epsilon_2)=(10^{-2},10^{-5})$&$\mathbf{2.34\pm0.03}$&-      &$5.42\pm0.06$ &$5.12\pm0.13$   &     -       &-\\
		$(\epsilon_1,\epsilon_2)=(10^{-5},10^{-5})$&$\mathbf{2.76\pm0.05}$ &-     &$38.93\pm0.71$&$14.90\pm0.16$   &   -       &-\\
		\hline
	\end{tabular}
	\caption{The mean and std of the number of iterations (when the stopping criteria is met) and the wall time of Meta-MgNet, PDE-MgNet and MG method on the testing set.  "-" means the algorithm does not converge within $10^4$ iterations.}\label{table:AnisotropicPoisson3de}
\end{table}

\subsection{Why is it Challenging to Train a Convergent PDE-MgNet for all $\eta$?}

For both the 2D and 3D anisotropic diffusion equations, we found it challenging to train appropriate weights for PDE-MgNet that can generalize beyond its training setting. This is, in fact, the motivation of viewing solving parameterized PDEs as multi-task learning and introducing Meta-MgNet. In this subsection, we conduct a simple experiment to demonstrate this issue with PDE-MgNet.

Consider the 3D anisotropic diffusion equation. We choose two different distributions for the parameters $D_1$: $\lg_2 {\epsilon_1}\sim U[-2,-1],\lg_2{\epsilon_2} \sim U[1,2]$ and  $D_2$: $\lg_2 {\epsilon_1} \sim U[1,2],\lg_2{\epsilon_2} \sim U[-2,-1]$ to train two PDE-MgNet. We present the weights of the convolution smoothers at the finest level ($\ell=1$), namely $\mathsf K^1$ in \eqref{kernel1} and $\mathsf K^2$ in \eqref{kernel2}: 
\begin{equation}\label{kernel1}
\mathsf K^1=\left[\begin{bmatrix}
-0.0170& -0.0503& -0.0169\\
 0.0051& \color{red}{-0.2370}&  0.0051\\
-0.0171& -0.0504& -0.0170
\end{bmatrix}
\begin{bmatrix}
 0.0389& \color{red}{0.2431}&  0.0389\\
-0.0696&  1.1127& -0.0695\\
 0.0388& \color{red}{0.2431}&  0.0389
\end{bmatrix}
\begin{bmatrix}
-0.0171& -0.0503& -0.0170\\
 0.0050& \color{red}{-0.2370}&  0.0051\\
-0.0170& -0.0503& -0.0169
\end{bmatrix}\right],
\end{equation}
\begin{equation}\label{kernel2}
\mathsf K^2=\left[\begin{bmatrix}
-0.0127& -0.0331& -0.0126\\
 0.0339& \color{red}{0.2175}&  0.0340\\
-0.0126& -0.0332& -0.0127
\end{bmatrix}
\begin{bmatrix}
 0.0113& \color{red}{-0.2041}&  0.0114\\
-0.0766&  1.0647& -0.0765\\
 0.0113& \color{red}{-0.2041}&  0.0113
\end{bmatrix}
\begin{bmatrix}
-0.0126& -0.0332& -0.0126\\
 0.0340&  \color{red}{0.2175}&  0.0340\\
-0.0127& -0.0331& -0.0127
\end{bmatrix}\right].
\end{equation}
Noting the red numbers in $\mathsf K^1$ and $\mathsf K^2$, we can see that $\mathsf K^1_{-1,0, 0}, \mathsf K^1_{1,0,0}< 0$ while $\mathsf K^2_{-1,0,0}, \mathsf K^2_{1,0,0} > 0$; and $ \mathsf K^1_{0,-1, 0}, \mathsf K^1_{0,1,0} > 0$ while  $\mathsf K^2_{0,-1 ,0},\mathsf K^2_{0,1,0} < 0$. Furthermore, if we use $K^1$ to smooth the PDEs with parameters generated from $D_2$, the error will increase, which means the convolution smoother with kernel $\mathsf K^1$ is not fit for $D_2$. We have the same issue with $\mathbf K_2$ and $D_1$. This phenomenon indicates that different distribution of parameters of the parameterized PDE may lead to weights of PDE-MgNet of contradictory effects. This is not only for PDE-MgNet, but rather an issue often occurs for supervised learning models. In contrast, Meta-MgNet handles this issue gracefully by adjusting the weights according to $\eta$ in an adaptive fashion.

\subsection{Ossen Equations}

The right-hand-side function $\utilde f$ in previous numerical experiments is randomly generated from the normal distribution. In this section, we include practical numerical examples to show the performance of Meta-MgNet. We adopt Ossen equations as an example:
\begin{equation*}
	\left\{\begin{aligned}
		-\mu \Delta \utilde u+(\utilde a \cdot \nabla)\utilde u + \nabla p &=  \utilde f,  &\text{in } \Omega,\\
		-\text{div} \utilde u &= \utilde 0,  &\text{in } \Omega,\\
		\utilde u &= \utilde 0,    &\text{on }  \partial \Omega,
	\end{aligned}\right.
\end{equation*}
where $\utilde u = (u,v)^\top$, $\mu = \dfrac 1{Re}$, $Re$ is the Reynold number, and $\utilde a = (a_x, a_y)^\top$.
Without loss of generality, we let $\mu = 1$.
We choose $\utilde u = \begin{pmatrix}  -2xy^2(1-x)(1-2x)(1-y)^2\\ 2x^2y(1-y)(1-2y)(1-x)^2\end{pmatrix}$ and $p = x^2-y^2$ as the solution and calculate the analytic form of the corresponding right-hand-side function $\utilde f$. The training data set is construct from sampling $a_x \sim U[0, 200]$ and $a_y\sim U[0,200]$. We use the MAC scheme \citep{nicolaides1996analysis, chen2015convergence} to discretize Ossen equations and the mesh size is $512\times 512$. Except for the settings mentioned above, other settings are the same as \textbf{Section 5.1}. Numerical solutions and error maps for a few different $\utilde a$ are presented in \textbf{Figure \ref{fig:ossen0}-\ref{fig:ossen3}} with the number of iterations and the wall time of Meta-MgNet recorded in the captions.

\begin{figure}[pos=htbp]
	\centering
	\includegraphics[width=0.9\linewidth, height=7cm]{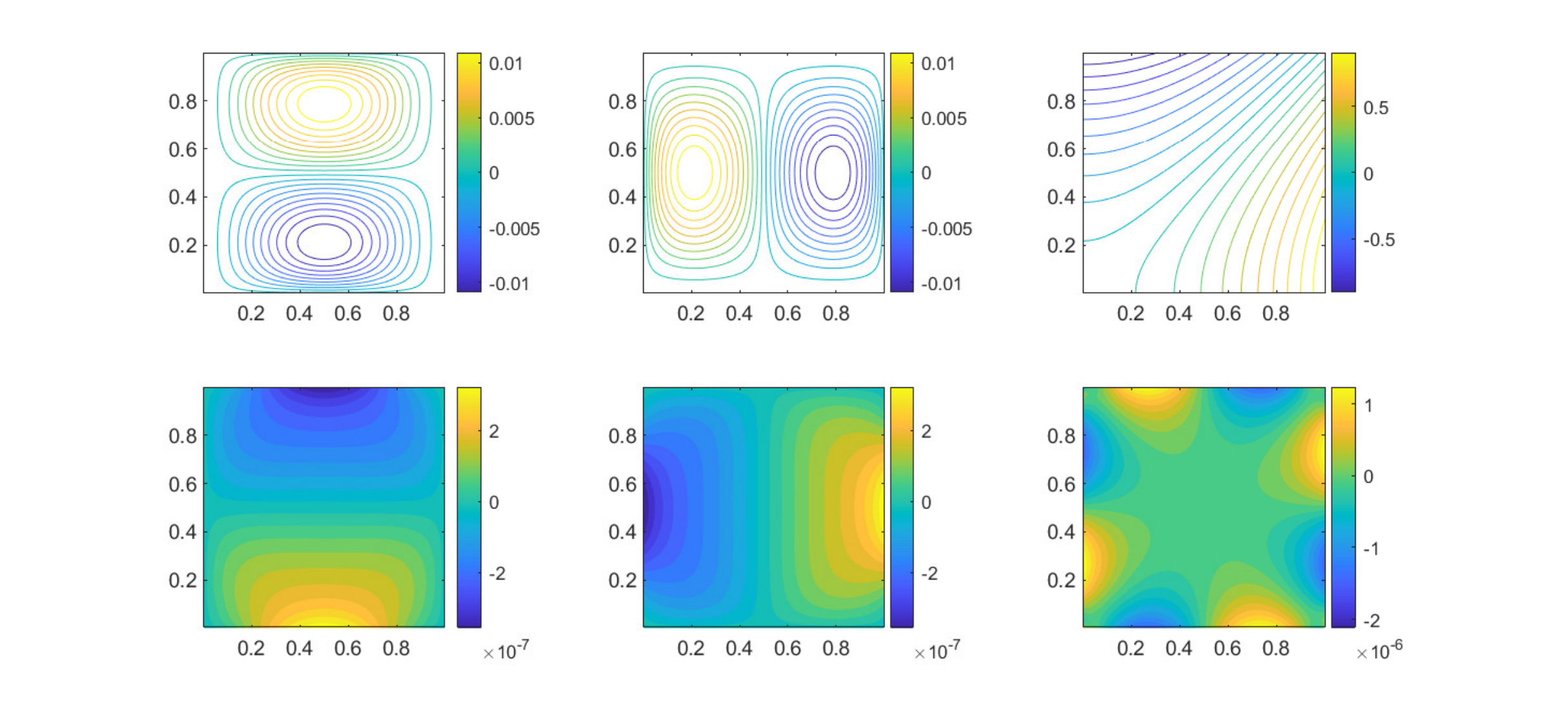}
	\caption{The three pictures on the top row are the numerical solutions and the three at the bottom row are the error of $\mathsf u$, $\mathsf v$ and $\mathsf p$ with $(a_x, a_y) = (0, 0)^\top$. The number of iterations is 367 and the wall time is 17.57s for Meta-MgNet.}
	\label{fig:ossen0}
\end{figure}
\begin{figure}[pos=htbp]
	\centering
	\includegraphics[width=0.9\linewidth, height=7cm]{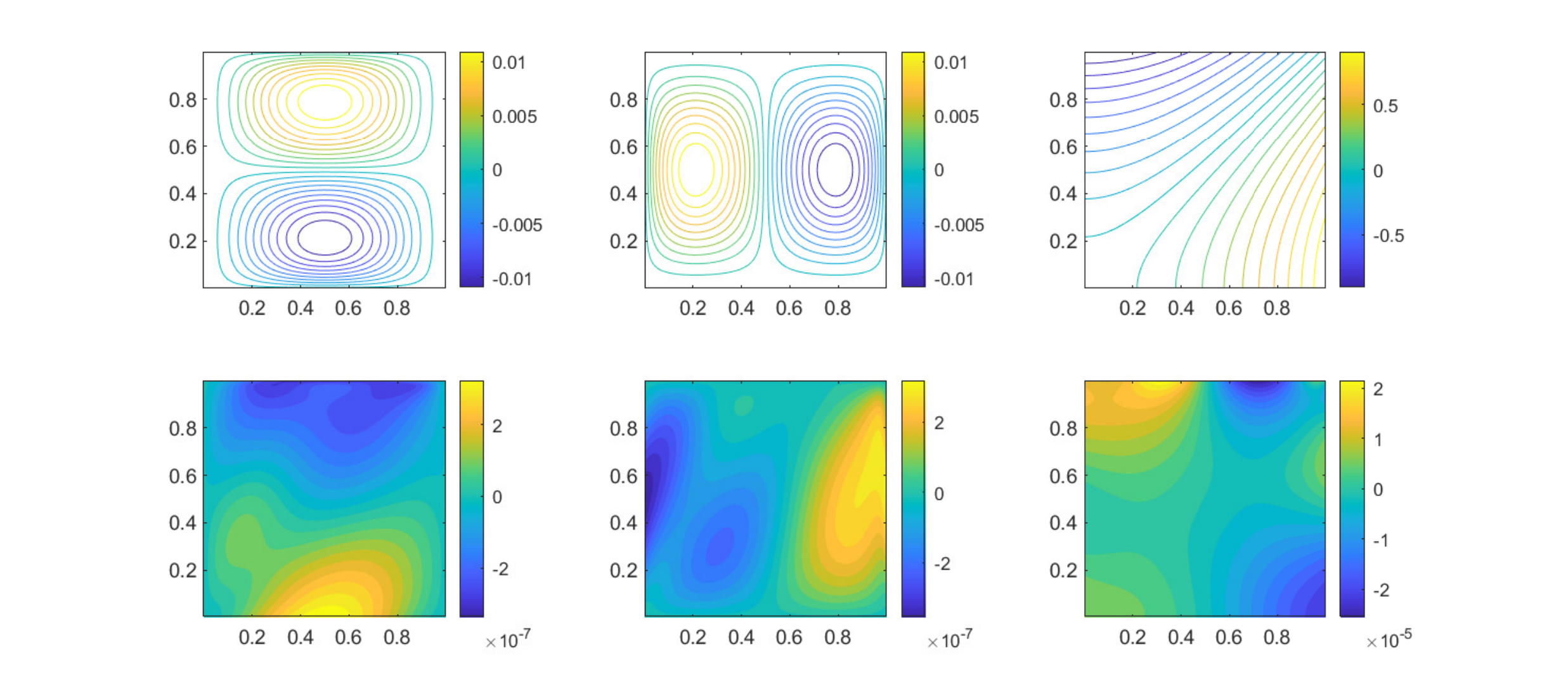}
	\caption{The three pictures on the top row are the numerical solutions and the three at the bottom row are the error of $\mathsf u$, $\mathsf v$ and $\mathsf p$ with $(a_x, a_y) = (50, 100)^\top$.  The number of iterations is 281 and the wall time is 13.6s for Meta-MgNet.}
	\label{fig:ossen1}
\end{figure}
\begin{figure}[pos=htbp]
	\centering
	\includegraphics[width=0.9\linewidth, height=7cm]{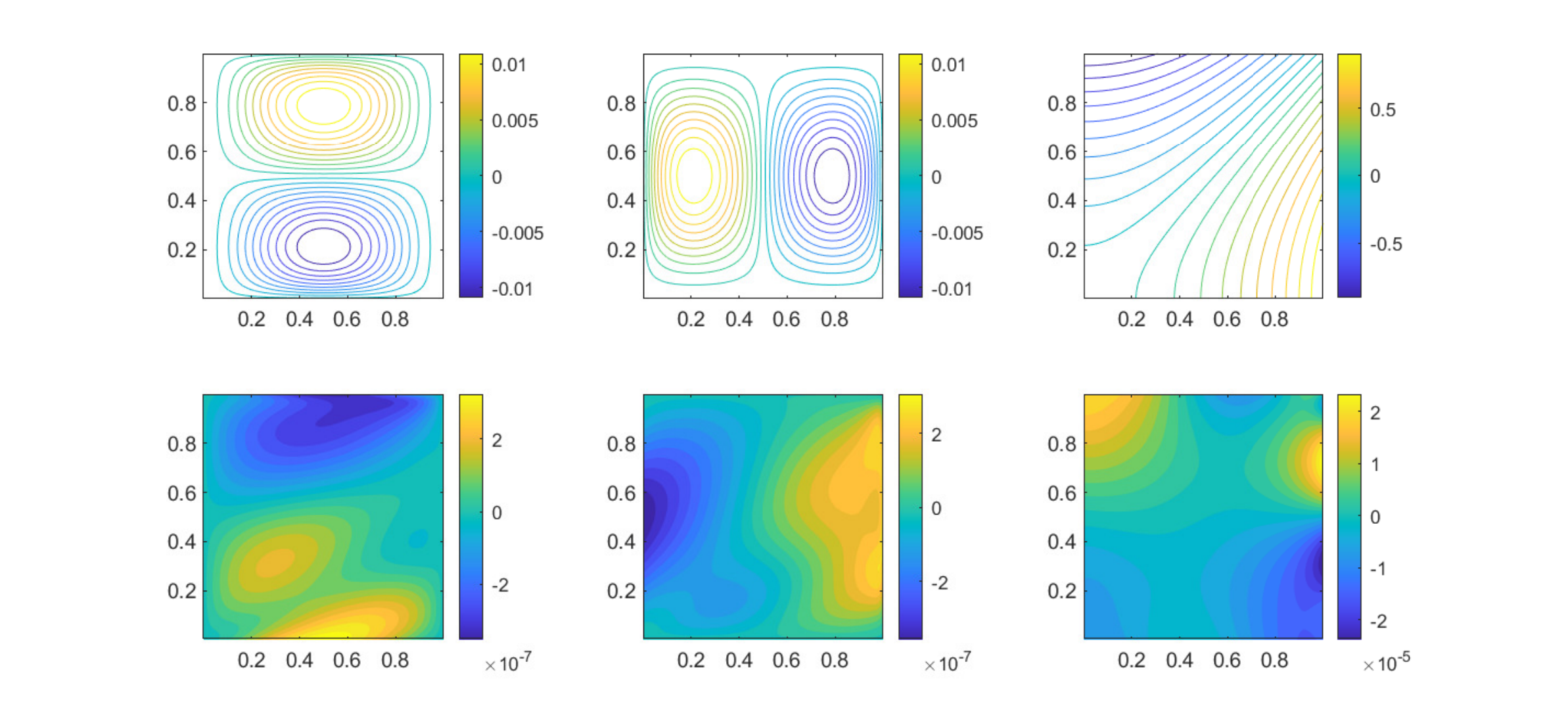}
	\caption{The three pictures on the top row are the numerical solutions and the three at the bottom row are the error of $\mathsf u$, $\mathsf v$ and $\mathsf p$ with $(a_x, a_y) = (100, 50)^\top$. The number of iterations is 275 and the wall time is 13.38s for Meta-MgNet.}
	\label{fig:ossen2}
\end{figure}
\begin{figure}[pos=htbp]
	\centering
	\includegraphics[width=0.9\linewidth, height=7cm]{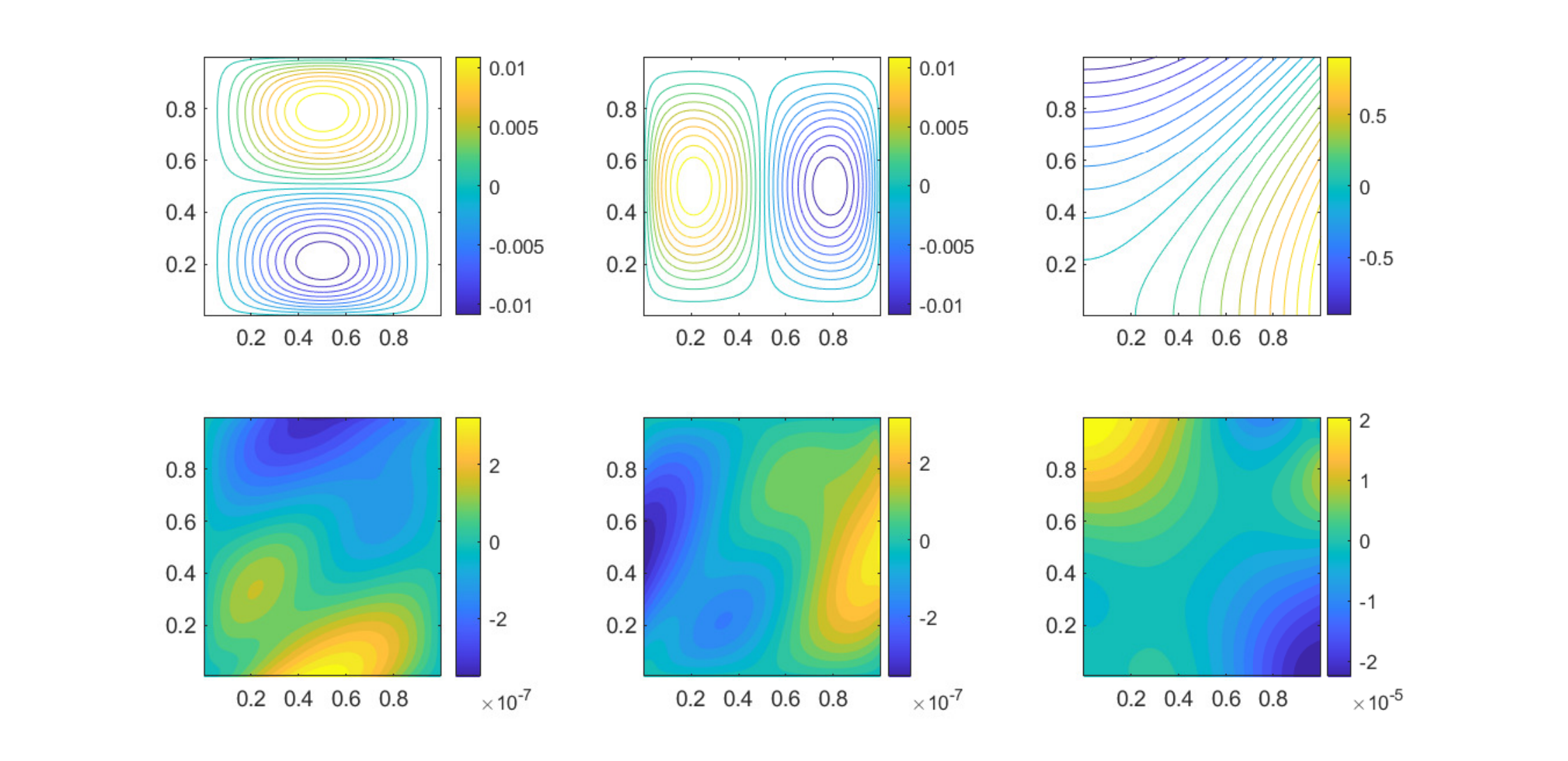}
	\caption{The three pictures on the top row are the numerical solutions and the three at the bottom row are the error of $\mathsf u$, $\mathsf v$ and $\mathsf p$ with $(a_x, a_y) = (100, 100)^\top$.  The number of iterations is 235 and the wall time is 11.52 for Meta-MgNet.}
	\label{fig:ossen3}
\end{figure}

\section{Conclusions and Future Work}
Solving parameterized PDEs is an essential and yet challenging task. In this paper, we provided a new perspective on the problem by viewing it as multi-task learning. With this, we proposed a new meta-learning based solver called Meta-MgNet by introducing a carefully designed hypernetwork (called Meta-NN) to the PDE-MgNet. Numerical experiments on 2D and 3D anisotropic diffusion equations showed that Meta-MgNet significantly outperforms the supervised learning-based PDE-MgNet and classical MG methods. Furthermore, Meta-MgNet manifested a clear advantage in training and generalization over PDE-MgNet, which demonstrated the feasibility of the proposed multi-task perspective and meta-learning approach to solving parameterized PDEs. 

This paper only discussed using meta-learning to improve smoothers because the prolongations and restrictions in classic MG methods are efficient enough to solve the three PDEs considered in this paper. As for some other PDEs such as Helmholtz equations, the convolutional prolongations and restrictions may not work well. Therefore, it is worth exploring a data-driven approach to improve prolongations and restrictions as well.  Furthermore, we only considered uniform mesh in this paper. We may consider generalizing MgNet or Meta-MgNet to nonuniform meshes, such as the triangular mesh, by exploiting tools from geometric deep learning, such as graph (convolutional) neural networks.

\section*{Acknowledgment}
Yuyan Chen was supported in part by the PSU-PKU Joint Center for Computational Mathematics and Applications, Bin Dong by the National Natural Science Foundation of China (grant No. 11831002), Beijing Natural Science Foundation (grant No. 180001), and Beijing Academy of Artificial Intelligence (BAAI), and Jinchao Xu by the Verne M. William Professorship Fund from the Pennsylvania State University and the National Science Foundation (grant No. DMS-1819157).

\bibliographystyle{cas-model2-names}
\bibliography{reference.bib}

\section{Appendix}
In the appendix, we add more numerical experiments to support some of our hyper-parameter choices. We use interpolation of 2D anisotropic diffusion equations as an example, and the setting of these experiments is the same as section 5.1.
\subsection{Appendix I}
The \textbf{ Table \ref{table:AppendixI}} shows the efficiency of fine-tuning for Meta-MgNet. We can find that the parameters inferred by meta smoother $B_\text{sc}$ are good enough so that the result is almost the same after fine-tuning. Therefore, we can skip the fine-tuning stage.
\begin{table}
	\begin{tabular}{|c|cc|}\hline
		\#iterations	&Meta-MgNet fine-tuning &Meta-MgNet      \\\hline
		$\epsilon=1$		    &$  3.0\pm0.00$         &$  4.0\pm0.00$  \\
		$\epsilon=10^{-1}$	    &$  7.0\pm0.00$         &$  7.5\pm0.50$  \\
		$\epsilon=10^{-2}$	    &$ 32.7\pm0.90$         &$ 35.1\pm1.04$  \\
		$\epsilon = 10^{-3}$    &$192.7\pm4.29$         &$171.6\pm6.34$  \\
		$\epsilon = 10^{-4}$    &$352.2\pm7.60$         &$375.2\pm5.88$  \\
		\hline
		wall time 				&                       &  	             \\\hline
		$\epsilon = 1$		    &$0.03\pm0.00$          &$0.03\pm0.00$   \\
		$\epsilon = 10^{-1}$    &$0.05\pm0.00$          &$0.05\pm0.00$   \\
		$\epsilon = 10^{-2}$    &$0.21\pm0.00$          &$0.22\pm0.01$   \\
		$\epsilon = 10^{-3}$    &$1.18\pm0.03$          &$1.06\pm0.04$   \\
		$\epsilon = 10^{-4}$    &$2.16\pm0.05$          &$2.31\pm0.03$   \\
		\hline
	\end{tabular}
	\caption{The mean and std of the number of iterations and the wall time of Meta-MgNet. No matter if Meta-MgNet is fine-tuned or not, the results are almost the same.}\label{table:AppendixI}
\end{table}
\subsection{Appendix II}
In section 5.1, we choose $(\nu_1,...,\nu_J)=(2,1,1,1,1)$. Now, we compare the result of several groups of $\nu_1,...,\nu_J$. Since it is easier to estimate error on coarse grid, to set $\nu_i > 1, i = 3,4,5$ is not necessary. Thus, we only test some pair of $\nu_1$ and $\nu_2$, and set $\nu_i = 1$ for $i = 3,4,5$.
The numerical experiments in \textbf{ Table \ref{table:AppendixII}} shows $(\nu_1,...,\nu_J)=(2,1,1,1,1)$ is a better choice. 

\begin{table}
	\begin{tabular}{|c|ccc|}\hline
		\#iterations	&Meta-MgNet&PDE-MgNet &PDE-MgNet-$\eta$  \\\hline
		$(\nu_1,...,\nu_J)=(1,1,1,1,1)$	&$42.1\pm1.58$ &$356.6\pm1.36$        &$230.60\pm6.99$       \\
		$(\nu_1,...,\nu_J)=(2,1,1,1,1)$	&$35.1\pm1.04$ &$178.9\pm2.74$&$ 149.7\pm 3.44$ \\
		$(\nu_1,...,\nu_J)=(3,1,1,1,1)$	&$21.9\pm1.81$  &$210.90\pm4.23$  & $198.60\pm6.89$       \\
		$(\nu_1,...,\nu_J)=(2,2,1,1,1)$	&$22.90\pm0.70$  &$214.50\pm2.42$      &$199.00\pm4.75$      \\
		\hline
		wall time 				&  &  &   	\\\hline
		$(\nu_1,...,\nu_J)=(1,1,1,1,1)$	&$0.23\pm0.01$ &$0.91\pm0.01$&$0.59\pm0.02$  \\
		$(\nu_1,...,\nu_J)=(2,1,1,1,1)$	&$0.22\pm0.01$ &$\mathbf{0.44\pm0.01}$&$\mathbf{0.37\pm0.01}$  \\
		$(\nu_1,...,\nu_J)=(3,1,1,1,1)$	&$\mathbf{0.16}\pm0.01$ &$0.56\pm0.01$  &$0.52\pm0.02$   \\
		$(\nu_1,...,\nu_J)=(2,2,1,1,1)$	&$0.17\pm0.00$ &$0.56\pm0.01$    &$0.53\pm0.01$\\
		\hline
	\end{tabular}
	\caption{The mean and std of the numbers of iteration and the wall time of different choices of $\nu_1,...,\nu_J$. The parameters of the 2D anisotropic diffusion equation are  $\epsilon=10^{-2}, \theta = 0$.}\label{table:AppendixII}
\end{table}

\begin{table}
	\begin{tabular}{|c|ccc|}\hline
		\#iterations	&Meta-MgNet&PDE-MgNet &PDE-MgNet-$\eta$  \\\hline
		$(\nu_1,...,\nu_J)=(1,1,1,1,1)$	&$1.2\text{e}3\pm22.51$ &-&$6.9\text{e}3\pm140.05$     \\
		$(\nu_1,...,\nu_J)=(2,1,1,1,1)$	&$375.2\pm5.88$ &-&$3.1\text{e}3\pm35.70$ \\
		$(\nu_1,...,\nu_J)=(3,1,1,1,1)$	&$355.30\pm10.51$ &-&$3.6\text{e}3\pm40.44$  \\
	    $(\nu_1,...,\nu_J)=(2,2,1,1,1)$	&$339.00\pm12.70$     &-& $2.8\text{e}3\pm33.22$        \\
		\hline
		wall time 				&  &  &   	\\\hline
		$(\nu_1,...,\nu_J)=(1,1,1,1,1)$	&$6.36\pm0.11$ &-& $17.52\pm0.34$      \\
		$(\nu_1,...,\nu_J)=(2,1,1,1,1)$	&$\mathbf{2.31\pm0.03}$ &-&$\mathbf{7.69\pm0.13}$ \\
		$(\nu_1,...,\nu_J)=(3,1,1,1,1)$	&$ 2.50\pm0.07$ &-&$9.42\pm0.11$   \\
		$(\nu_1,...,\nu_J)=(2,2,1,1,1)$	& $2.34\pm0.09$   &-&$8.45\pm0.10$     \\
		\hline
	\end{tabular}
	\caption{The mean and std of the number of iterations and the wall time of different choices of $\nu_1,...,\nu_J$. The parameters of the 2D anisotropic diffusion equation are  $\epsilon=10^{-4}, \theta = 0$.}\label{table:AppendixI2}
\end{table}
\subsection{Appendix III}
We compare the efficiency between $B_{\text d}$ in \textbf{Algorithm \ref{alg:Bd}} and $B_{\text {sc}}$ in \textbf{ Table \ref{table:AppendixIII}}. As we can see, $B_{\text {sc}}$ is much better than $B_{\text d}$. Therefore, we choose $B_{\text {sc}}$ as the meta smoother in other numerical experiments in this paper.
\begin{table}
	\begin{tabular}{|c|cc|}\hline
		\#iterations&$B_{\text {d}}$    &$B_{\text {sc}}$  \\\hline
		$\epsilon=1$		&- 					&$  4.0\pm0.00$	   \\
		$\epsilon=10^{-1}$	&$58.9\pm0.30$ 		&$  7.5\pm0.50$	   \\
		$\epsilon=10^{-2}$	&$597.8\pm8.08$ 	&$ 35.1\pm1.04$    \\
		$\epsilon = 10^{-3}$&$5.5\text{e}3\pm60.89$	&$171.6\pm6.34$    \\
		$\epsilon = 10^{-4}$&-				    &$375.2\pm5.88$    \\
		\hline
		wall time 				&  &     	\\\hline
		$\epsilon = 1$		&-					&$0.03\pm0.00$ 	\\
		$\epsilon = 10^{-1}$&$0.21\pm0.00$      &$0.05\pm0.00$  \\
		$\epsilon = 10^{-2}$&$2.07\pm0.03$  	&$0.22\pm0.01$ 	\\
		$\epsilon = 10^{-3}$&$19.04\pm0.20$    	&$1.06\pm0.04$  \\
		$\epsilon = 10^{-4}$&-					&$2.31\pm0.03$  \\
		\hline
	\end{tabular}
	\caption{The mean and std of the number of iterations and the wall time for the convolutional smoother $B_{\text {d}}$ and the SC method $B_{\text {sc}}$. }\label{table:AppendixIII}
\end{table}
\end{document}